\documentclass[12pt,twoside]{amsart}

\usepackage{amsmath,amsthm,amscd,amssymb,mathrsfs,graphicx,amsfonts,mathrsfs}
\usepackage{amsfonts}
\usepackage{amssymb,enumerate}
\usepackage{amsthm}
\usepackage[all]{xy}
\usepackage{hyperref}
\allowdisplaybreaks[4] \footskip=12pt
\renewcommand{\uppercasenonmath}[1]{}

\numberwithin{equation}{section} \theoremstyle{plain}
\newtheorem*{thm*}{Main Theorem}
\newtheorem{thm}{Theorem}[section]
\newtheorem{cor}[thm]{Corollary}
\newtheorem*{cor*}{Corollary}
\newtheorem{lem}[thm]{Lemma}
\newtheorem*{lem*}{Lemma}

\newtheorem*{fact*}{Fact}

\newtheorem*{nota*}{Notation}
\newtheorem{prop}[thm]{Proposition}
\newtheorem*{prop*}{Proposition}
\newtheorem{rem}[thm]{Remark}
\newtheorem*{rem*}{Remark}

\newtheorem*{observation*}{Observation}

\newtheorem*{exa*}{Example}

\newtheorem*{df*}{Definition}

\newtheorem*{con*}{Construction}

\renewcommand{\geq}{\geqslant}
\renewcommand{\leq}{\leqslant}

\begin{document}
\begin{center}
{\large  \bf Intersection Theorem for DG-modules}

\vspace{0.5cm} Xiaoyan Yang\\
School of Science, Zhejiang University of Science and Technology, Hangzhou 310023, China\\
E-mail: yangxy@zust.edu.cn
\end{center}

\bigskip
\centerline { \bf  Abstract}
\leftskip10truemm \rightskip10truemm \noindent
Let $A$ be a commutative noetherian local DG-ring with bounded cohomology. The Intersection Theorem for DG-modules is examined and some of its applications are provided. The first is to prove the DG-setting of the amplitude inequality, New Intersection Theorem and Krull's principle ideal theorem. The second is to solve completely the Minamoto's conjecture in [Israel J. Math. 242 (2021) 1--36]. The third is to show the DG-version of the Bass conjecture about Cohen-Macaulay
rings and the Vasconcelos conjecture about Gorenstein rings.
\leftskip10truemm \rightskip10truemm \noindent
\\[2mm]

\leftskip0truemm \rightskip0truemm
\section{\bf Introduction}\label{pre}
The New Intersection Theorem (NIT) of Peskine and Szpiro \cite{PS}, Hochster \cite{H} and Roberts \cite{R} is a central result in the homological theory of commutative
noetherian rings, which has a variety of applications (see \cite{F,SY}), and was
extended and polished by commutative algebra practitioners (see \cite{AIN,DY1,FI,I,IMSW,PW}).

Let $A$ be a  commutative noetherian local ring and $M,N$ two finitely generated $A$-modules with $\mathrm{projdim}_AM<\infty$.
 The NIT yields an inequality, i.e., Intersection Theorem (IT)
 \begin{align}
\mathrm{dim}_AN\leq \mathrm{dim}_A(M\otimes_AN)+\mathrm{projdim}_AM
\label{exact03}
\tag{$\dag$}\end{align}
 By applying the inequality $(\dag)$ with $N=A$, one has
$\mathrm{dim}A\leq \mathrm{dim}_AM+\mathrm{projdim}_AM$.
By the Auslander-Buchsbaum Formula, the inequality is equivalent to
$\mathrm{cmd}A\leq \mathrm{cmd}_AM$,
where $\mathrm{cmd}_AM=\mathrm{dim}_AM-\mathrm{depth}_AM$ is the Cohen-Macaulay defect of $M$ which determines the failure of $M$ to be Cohen-Macaulay. Further,
Yassemi \cite{Ya} showed the inequalities
\begin{align}
\mathrm{dim}_AN\leq\mathrm{dim}_A\mathrm{RHom}_A(M,N)\leq\mathrm{dim}_A(M\otimes_AN)+\mathrm{projdim}_AM
\label{exact03}
\tag{$\ddag$}\end{align}
By \cite[Proposition 5.2.6]{CF},
$\mathrm{depth}_A\mathrm{RHom}_A(M,A)=\mathrm{depth}A$, it follows by $(\ddag)$ that
$\mathrm{cmd}A\leq\mathrm{cmd}_A\mathrm{RHom}_A(M,A)\leq\mathrm{cmd}_AM$. Also
$\mathrm{grade}_AM+\mathrm{dim}_AM\leq\mathrm{dim}A$, so
if $M$ is perfect then $\mathrm{cmd}A=\mathrm{cmd}_A\mathrm{RHom}_A(M,A)=\mathrm{cmd}_AM$.

On the other hand, Foxby \cite{F} generalized the NIT for complexes, for
bounded complexes $X,Y$ with finite cohomologies, showed the IT of complexes \begin{center}
$\mathrm{dim}_AY\leq\mathrm{dim}_A(X\otimes^\mathrm{L}_AY)+\mathrm{projdim}_AX$\end{center}whenever $\mathrm{projdim}_AX<\infty$. In 2005, Dibaei and Yassemi \cite{DY1} continued to study the IT of complexes and proved the following inequalities\begin{center}
$\mathrm{dim}_AY-\mathrm{sup}X\leq\mathrm{dim}_A\mathrm{RHom}_A(X,Y)
\leq\mathrm{dim}_A(X\otimes^\mathrm{L}_AY)+\mathrm{projdim}_AX-\mathrm{inf}X$,
\end{center}
which generalizes the above inequalities $(\ddag)$.

 An \emph{associative DG-ring} $A$ is a $\mathbb{Z}$-graded ring
$A=\bigoplus_{i\in\mathbb{Z}}A^i$ with a diferential
 of degree 1, that satisfes the graded Leibniz rule.
 A DG-ring $A$ is
called \emph{commutative} if $b\cdot a=(-1)^{i\cdot j}a\cdot b$ for all $a\in A^i$ and $b\in A^j$, and $a^2=0$
if $i$ is odd.
A DG-ring $A$ is called \emph{non-positive} if $A^i=0$ for all $i>0$.
 A non-positive DG-ring $A$ is called \emph{noetherian} if the
ring $\mathrm{H}^0(A)$ is noetherian and the $\mathrm{H}^0(A)$-module $\mathrm{H}^i(A)$ is
finitely generated for all $i<0$.
If $A$ is a noetherian DG-ring and $(\mathrm{H}^0(A),\bar{\mathfrak{m}},\bar{\kappa})$ is a local ring, then we say that $(A,\bar{\mathfrak{m}},\bar{\kappa})$ is \emph{local noetherian}. \textbf{Unless stated to the contrary we assume throughout this paper that $A$ is a commutative and noetherian DG-ring.}

The derived category of DG-modules over $A$ will be denoted by $\mathrm{D}(A)$, and write $-\otimes_A^\mathrm{L}-$
 for its monoidal product, $\mathrm{RHom}_A(-,-)$ for the internal hom. The full subcategory which consists of DG-modules
whose cohomology is above (resp. below, bounded) is denoted by $\mathrm{D}^-(A)$ (resp.
$\mathrm{D}^+(A)$, $\mathrm{D}^\mathrm{b}(A)$), denote by $\mathrm{D}_\mathrm{f}(A)$ the full subcategory of $\mathrm{D}(A)$
consisting of DG-modules with finite cohomology, and set
$\mathrm{D}^-_\mathrm{f}(A)=\mathrm{D}^-(A)\cap \mathrm{D}_\mathrm{f}(A)$, $\mathrm{D}^+_\mathrm{f}(A)=\mathrm{D}^+(A)\cap \mathrm{D}_\mathrm{f}(A)$, $\mathrm{D}^\mathrm{b}_\mathrm{f}(A)=\mathrm{D}^\mathrm{b}(A)\cap \mathrm{D}_\mathrm{f}(A)$.
  For a DG-module $X$, we set $\mathrm{inf}X:=\mathrm{inf}\{n\in\mathbb{Z}\hspace{0.03cm}|\hspace{0.03cm}\mathrm{H}^n(X)\neq0\}$,
$\mathrm{sup}X:=\mathrm{sup}\{n\in\mathbb{Z}\hspace{0.03cm}|\hspace{0.03cm}\mathrm{H}^n(X)\neq0\}$
and $\mathrm{amp}X:=\mathrm{sup}X-\mathrm{inf}X$.

DG-rings allow
us to use techniques of homological algebra of ordinary rings in a much wider context.
Let $(A,\bar{\mathfrak{m}},\bar{\kappa})$ be a local DG-ring. Shaul \cite{Shaul} extended the definitions of Krull dimensions and depth of complexes to DG-setting. The \emph{local cohomology Krull dimension}
of a DG-module $X\in\mathrm{D}^-(A)$ is the number
\begin{center}$\mathrm{lc.dim}_AX:=
\mathrm{sup}\{\mathrm{dim}_{\mathrm{H}^0(A)}\mathrm{H}^\ell(X)+\ell\hspace{0.03cm}|\hspace{0.03cm}\ell\in\mathbb{Z}\}$,\end{center}where $\mathrm{dim}_{\mathrm{H}^0(A)}\mathrm{H}^\ell(X)$ is the usual Krull dimension of the $\mathrm{H}^0(A)$-module $\mathrm{H}^\ell(X)$.
 The \emph{depth} of a DG-module $Y\in\mathrm{D}^+(A)$ is the number
\begin{center}$\mathrm{depth}_AY:=\mathrm{inf}\mathrm{RHom}_A(\bar{\kappa},Y)$.\end{center}
The theory of classical Cohen-Macaulay rings, Cohen-Macaulay modules and Gorenstein rings has also been extended in \cite{Shaul, sh20}. This paper focuses on local cohomological Krull dimension of the DG-modules $\mathrm{RHom}_A(X,Y)$ and $X\otimes_A^\mathrm{L}Y$ for appropriate $X$ and $Y$ in $\mathrm{D}(A)$. We wish to extend the above inequalities to local DG-rings and obtains some characterizations about local Cohen-Macaulay and local Gorenstein DG-rings.

The paper is organized as follows.

In Section 2, we gather
some preliminaries and results about local cohomological Krull dimension of DG-modules, and generalize the DG-version of Grothendieck's
local duality theorem in \cite[Theorem 7.26]{s18} to a more general case. In Section 3,
 the DG-setting of IT is proved when $A$ is a local DG-ring with bounded cohomology (see Theorem \ref{lem5.4}), a bit more generally, for $0\not\simeq Y\in\mathrm{D}^{-}_{\mathrm{f}}(A)$ and $X\in\mathrm{D}^-(A)$ with $\mathrm{flatdim}_AX<\infty$ and
$\mathrm{depth}_AX<\infty$, we show the following inequality:
\begin{align}
\mathrm{lc.dim}_AY\leq\mathrm{lc.dim}_A(X\otimes^\mathrm{L}_AY)+\mathrm{flatdim}_AX
\label{exact03}
\tag{$\dag'$}\end{align}
Moreover, some upper and lower bounds of
$\mathrm{lc.dim}_{A}(X\otimes_A^\mathrm{L}Y)$ and $\mathrm{lc.dim}_{A}\mathrm{RHom}_A(X,Y)$ for some DG-modules $X$ and $Y$ are  provided (see Theorem \ref{lem5.4'}). When the DG-ring $A$ is sequence-regular and  $X,Y\in\mathrm{D}^{\mathrm{b}}_\mathrm{f}(A)$, we improve the inequality $(\dag')$ in Proposition \ref{lem5.7}, prove the following stronger inequalities:
\begin{center}
$\mathrm{lc.dim}_AX+\mathrm{lc.dim}_AY\leq\mathrm{lc.dim}A+\mathrm{lc.dim}_A(X\otimes^\mathrm{L}_AY)$.
\end{center}

In view of the inequalities $(\dag')$, we extend in Section 4 some classical results about amplitude inequality, Cohen-Macaulay rings and Gorenstein rings to local DG-rings with bounded cohomology. For $0\not\simeq X,Y\in\mathrm{D}^{\mathrm{b}}_{\mathrm{f}}(A)$ with $\mathrm{projdim}_AX<\infty$,  J${\o}$rgensen \cite[Theorem 3.1]{J} showed the amplitude inequality $\mathrm{amp}(X\otimes^\mathrm{L}_AY)\geq\mathrm{amp}X$ when $A$ is a sufficiently nice local DG-ring with $\mathrm{amp}A<\infty$. The first application of our results is to eliminate the hypothesis on $A$. We improve the J${\o}$rgensen's result (see Proposition \ref{lem5.10}), and prove the DG-version of NIT and Krull's principle ideal theorem (see Proposition \ref{lem1.4}). Bass \cite{B} raised the question
  whether existence of a finitely generated module
of finite injective dimension would imply Cohen-Macaulayness of the ring, and later known as the Bass conjecture.  For $0\not\simeq X\in\mathrm{D}^{\mathrm{b}}_\mathrm{f}(A)$ with $\mathrm{injdim}_AX<\infty$, Minamoto \cite[Conjecture 2.36]{Mi19} conjectured that
$\mathrm{amp}X\geq \mathrm{amp}A$.
Shaul \cite[Theorem 5.22(2)]{Shaul} showed the DG-version of Bass Conjecture and
 solved this conjecture under the
noetherian model assumption, and conjectured in \cite[Remark 5.25]{Shaul} that this assumption is redundant.
  As the second application, Minamoto's Conjecture or Shaul's Conjecture is solved (see Proposition \ref{lem5.9}).
  Vasconcelos \cite{V} conjectured that rings of type one are Gorenstein. The last application is to prove the DG-version of the Vasconcelos conjecture (see Proposition \ref{lem7.7}).

\bigskip
\section{\bf Preliminaries and some facts}
\vspace{1.5mm}
This section is devoted to recalling some notions and basic facts.

\vspace{1.5mm}
{\bf
Projective, injective and flat dimensions.} For a DG-module $X\in \mathrm{D}(A)$,
we recall the definition of the projective, injective and flat dimensions of $X$ introduced by Bird, Shaul, Sridhar and Williamson in \cite{BSSW}, which are different from the definition, $\mathrm{pd}X$, $\mathrm{id}X$ and $\mathrm{fd}X$, introduced by Yekutieli \cite{ye16}.

The \emph{projective dimension} of $X$ is defined by
\begin{center}$\mathrm{projdim}_AX=\mathrm{inf}\{n\in\mathbb{Z}\hspace{0.03cm}|\hspace{0.03cm}\mathrm{H}^i(\mathrm{RHom}_A(X,Y))=0\ \textrm{for\ any}\ Y\in\mathrm{D}^\mathrm{b}(A)\ \textrm{and}\ i>n+\mathrm{sup}Y\}$.\end{center}

The \emph{injective dimension} of $X$ is defined by
\begin{center}$\mathrm{injdim}_AX=\mathrm{inf}\{n\in\mathbb{Z}\hspace{0.03cm}|\hspace{0.03cm}\mathrm{H}^i(\mathrm{RHom}_A(Y,X))=0\ \textrm{for\ any}\ Y\in\mathrm{D}^\mathrm{b}(A)\ \textrm{and}\ i>n-\mathrm{inf}Y\}$.\end{center}

The \emph{flat dimension} of $X$ is defined by
\begin{center}$\mathrm{flatdim}_AX=\mathrm{inf}\{n\in\mathbb{Z}\hspace{0.03cm}|\hspace{0.03cm}\mathrm{H}^{-i}(Y\otimes_A^\mathrm{L}X)=0\ \textrm{for\ any}\ Y\in\mathrm{D}^\mathrm{b}(A)\ \textrm{and}\ i>n-\mathrm{inf}Y\}$.\end{center}

For $X\in\mathrm{D}^-(A)$ and $Y\in\mathrm{D}^+(A)$, it follows by \cite[Theorems 2.22, 3.21, 4.13]{Mi18} that \begin{center}$\mathrm{projdim}_AX=\mathrm{pd}X-\mathrm{sup}X$,\end{center}\begin{center}$\mathrm{injdim}_AY=\mathrm{id}Y+\mathrm{inf}Y$,\end{center}
\begin{center}$\mathrm{flatdim}_AX=\mathrm{fd}X-\mathrm{sup}X$.\end{center}

\begin{lem}\label{lem0.11} Let $A$ be a DG-ring with $\mathrm{amp}A<\infty$, $X\in\mathrm{D}^{-}(A)$ and $Y\in\mathrm{D}^{+}(A)$.

$(1)$ If $\mathrm{projdim}_AX<\infty$, then $\mathrm{inf}X\geq\mathrm{inf}A-\mathrm{projdim}_AX$, i.e. $X\in\mathrm{D}^{\mathrm{b}}(A)$.

 $(2)$ If $\mathrm{injdim}_AY<\infty$, then $\mathrm{sup}Y\leq-\mathrm{inf}A+\mathrm{injdim}_AY$, i.e. $Y\in\mathrm{D}^{\mathrm{b}}(A)$.

 $(3)$ If $\mathrm{flatdim}_AX<\infty$, then $\mathrm{inf}X\geq\mathrm{inf}A-\mathrm{flatdim}_AX$, i.e. $X\in\mathrm{D}^{\mathrm{b}}(A)$.
\end{lem}
\begin{proof} We just prove (1) since (2) follows by duality.

(1) We use induction on $n=\mathrm{projdim}_AX+\mathrm{sup}X$. If $n=0$, then $X\in\mathcal{P}[-\mathrm{sup}X]$ by \cite[Lemma 2.14]{Mi18} and so $\mathrm{inf}X\geq\mathrm{inf}A-\mathrm{projdim}_AX$. Assume that $n\geq1$. By \cite[Lemma 2.13]{Mi18}, there is an exact triangle $Z\rightarrow P\rightarrow X\rightsquigarrow$ so that $P\in\mathcal{P}[-\mathrm{sup}X]$ and $\mathrm{projdim}_AZ+\mathrm{sup}Z\leq\mathrm{projdim}_AZ+\mathrm{sup}X=n-1$, it follows by induction that $\mathrm{inf}X\geq\mathrm{inf}A-\mathrm{projdim}_AX$.

(3) For $\mathrm{H}^0(A)$-module $\mathrm{Hom}_{\mathbb{Z}}(\mathrm{H}^0(A),\mathbb{Q}/\mathbb{Z})$, it follows by \cite[Theorem 5.7]{s18} that there is a DG-module $E$ so that $\mathrm{H}^0(E)=\mathrm{Hom}_{\mathbb{Z}}(\mathrm{H}^0(A),\mathbb{Q}/\mathbb{Z})$. Hence \cite[Theorem 4.10]{s18} and (2) imply that $\mathrm{inf}X=-\mathrm{sup}\mathrm{RHom}_A(X,E)\geq\mathrm{inf}A-\mathrm{injdim}_A\mathrm{RHom}_A(X,E)=\mathrm{inf}A-\mathrm{flatdim}_AX$.
\end{proof}

\begin{lem}\label{lem0.15} {\rm (\cite{BSSW}).} Let $A$ be a DG-ring with $\mathrm{amp}A<\infty$ and $X,Y,Z\in\mathrm{D}(A)$.

$(1)$ If either $X\in\mathrm{D}^{-}_{\mathrm{f}}(A),\ Y\in\mathrm{D}^{+}(A)$ and $\mathrm{injdim}_AZ<\infty$, or $X\in\mathrm{D}^{\mathrm{b}}_{\mathrm{f}}(A)$ and $\mathrm{projdim}_AX<\infty$, then the
natural map \begin{center}$X\otimes^\mathrm{L}_A\mathrm{RHom}_A(Y,Z)\rightarrow\mathrm{RHom}_A(\mathrm{RHom}_A(X,Y),Z)$\end{center}
is an isomorphism in $\mathrm{D}(A)$.

$(2)$ If either $X\in\mathrm{D}^{-}_{\mathrm{f}}(A),\ Y\in\mathrm{D}^{+}(A)$ and $\mathrm{flatdim}_AZ<\infty$, or $X\in\mathrm{D}^{\mathrm{b}}_{\mathrm{f}}(A)$ and
 $\mathrm{projdim}_AX<\infty$, then the
natural map \begin{center}$\mathrm{RHom}_A(X,Y)\otimes^\mathrm{L}_AZ\rightarrow\mathrm{RHom}_A(X,Y\otimes^\mathrm{L}_AZ)$\end{center}
is an isomorphism in $\mathrm{D}(A)$.
\end{lem}

\vspace{1.5mm}
{\bf Derived torsion and completion.} Let $\bar{\mathfrak{a}}\subseteq\mathrm{H}^0(A)$ be an ideal. Recall that
an $\mathrm{H}^0(A)$-module $\bar{X}$ is called \emph{$\bar{\mathfrak{a}}$-torsion} if for any $\bar{x}\in\bar{X}$ there exists $n\in\mathbb{N}$ such that $\bar{\mathfrak{a}}^n\bar{x}=0$.
The category $\mathrm{D}_{\bar{\mathfrak{a}}\textrm{-tor}}(A)$ consisting of DG-modules $X$ such that
 the $\mathrm{H}^0(A)$-module $\mathrm{H}^n(X)$ is $\bar{\mathfrak{a}}$-torsion for every $n$, is a triangulated subcategory
of $\mathrm{D}(A)$. Following \cite{s19}, the inclusion functor
$\mathrm{D}_{\bar{\mathfrak{a}}\textrm{-tor}}(A)\hookrightarrow\mathrm{D}(A)$
has a right adjoint
$\mathrm{D}(A)\rightarrow\mathrm{D}_{\bar{\mathfrak{a}}\textrm{-tor}}(A)$,
and composing this right adjoint with the inclusion, one obtains a triangulated
functor
\begin{center}$\mathrm{R}\Gamma_{\bar{\mathfrak{a}}}:\mathrm{D}(A)\rightarrow\mathrm{D}(A)$,\end{center}
which is called the \emph{derived $\bar{\mathfrak{a}}$-torsion functor}.
The \emph{derived
$\bar{\mathfrak{a}}$-completion functor} $\mathrm{L}\Lambda_{\bar{\mathfrak{a}}}:\mathrm{D}(A)\rightarrow\mathrm{D}(A)$ is defined by $\mathrm{L}\Lambda_{\bar{\mathfrak{a}}}(X):=\mathrm{RHom}_A(\mathrm{R}\Gamma_{\bar{\mathfrak{a}}}(A),X)$, which is left adjoint to $\mathrm{R}\Gamma_{\bar{\mathfrak{a}}}$ as endofunctors on $\mathrm{D}(A)$. The derived $\bar{\mathfrak{a}}$-adic completion $\mathrm{L}\Lambda_{\bar{\mathfrak{a}}}(A)$, denoted by $\mathrm{L}\Lambda(A,\bar{\mathfrak{a}})$, is itself a commutative noetherian DG-ring.

Given a commutative ring $A$ and $a\in A$, the telescope complex
associated to $A$ and $a$ is
\begin{center}$0\rightarrow\bigoplus_{n=0}^\infty A\rightarrow\bigoplus_{n=0}^\infty A\rightarrow0$\end{center}
in degrees 0, 1, with the differential being defined by
\begin{center}$d(e_i)=\left\{\begin{array}{ll}
     e_0\ \ \ \ \ \ \ \ \ \ \ \ \ \  i=0, \\
    e_{i-1}-ae_i\ \ \ \ i\geq1,
   \end{array}\right.$\end{center}
where $e_0,e_1,\cdots$ denote the standard basis of the countably generated
free $A$-module $\bigoplus_{n=0}^\infty A$.
We denote this complex by $\mathrm{Tel}(A;a)$.
Given a
sequence $\emph{\textbf{a}}=a_1,\cdots,a_n\in A$, we set
\begin{center}$\mathrm{Tel}(A;\emph{\textbf{a}}):=\mathrm{Tel}(A;a_1)\otimes_A\cdots\otimes_A\mathrm{Tel}(A;a_n)$,\end{center}
which is a bounded complex of free $A$-modules, called the telescope
complex associated to $A$ and $\emph{\textbf{a}}$.

For an ideal $\bar{\mathfrak{a}}\subseteq\mathrm{H}^0(A)$,
if $\mathfrak{a}= (a_1,\cdots,a_n)$ is a
sequence in $A^0$, whose image in $\mathrm{H}^0(A)$ generates $\bar{\mathfrak{a}}$, then for $X\in\mathrm{D}(A)$, there are
natural isomorphisms
\begin{center}$\mathrm{R}\Gamma_{\bar{\mathfrak{a}}}(X)\cong\mathrm{Tel}(A^0;\mathfrak{a})\otimes_{A^0}A\otimes_AX$, $\mathrm{L}\Lambda_{\bar{\mathfrak{a}}}(X)\cong\mathrm{Hom}_{A}(\mathrm{Tel}(A^0;\mathfrak{a})\otimes_{A^0}A,X)$.\end{center}

Let $(A,\bar{\mathfrak{m}})$ be a local DG-ring. For $X\in\mathrm{D}^-_\mathrm{f}(A)$ and $Y\in\mathrm{D}^+(A)$, it follows by \cite[Theorem 2.15]{Shaul} and \cite[Proposition 3.3]{Shaul} that $\mathrm{lc.dim}_AX=\mathrm{supR}\Gamma_{\bar{\mathfrak{m}}}(X)$ and $\mathrm{depth}_AY=\mathrm{infR}\Gamma_{\bar{\mathfrak{m}}}(Y)$.

\vspace{1.5mm}
{\bf Localization.} Let $\pi_A:A\rightarrow\mathrm{H}^0(A)$ be the canonical surjection and $\pi^0_A: A^0\rightarrow\mathrm{H}^0(A)$ be its degree 0 component. Given a prime ideal
$\bar{\mathfrak{p}}\in \mathrm{Spec}\mathrm{H}^0(A)$, let $\mathfrak{p}=(\pi^0_A)^{-1}(\bar{\mathfrak{p}})\in\mathrm{Spec}A^0$, and define $A_{\bar{\mathfrak{p}}}:=A\otimes_{A^0}A^0_\mathfrak{p}$. More generally, given $X\in\mathrm{D}(A)$, we define
\begin{center}$X_{\bar{\mathfrak{p}}}:=X\otimes_{A}A_{\bar{\mathfrak{p}}}=X\otimes_{A^0}A^0_\mathfrak{p}\in\mathrm{D}(A_{\bar{\mathfrak{p}}})$.\end{center} A prime ideal $\bar{\mathfrak{p}}\in\mathrm{Spec}\mathrm{H}^0(A)$ is called an \emph{associated
prime} of $X$ if $\mathrm{depth}_{A_{\bar{\mathfrak{p}}}}X_{\bar{\mathfrak{p}}}=\mathrm{inf}X_{\bar{\mathfrak{p}}}$. The set of associated primes of $X$ is denoted by $\mathrm{Ass}_AX$. The \emph{support} of $X$ is the set
\begin{center}$\mathrm{Supp}_AX:=\{\bar{\mathfrak{p}}\in \mathrm{Spec}\mathrm{H}^0(A)\hspace{0.03cm}|\hspace{0.03cm}X_{\bar{\mathfrak{p}}}\not\simeq0\}$.\end{center}
For $\bar{\mathfrak{p}}\in \mathrm{Spec}\mathrm{H}^0(A)$, denote by $E(A,\bar{\mathfrak{p}})$ the DG-module corresponding to the
injective hull $E(\mathrm{H}^{0}(A),\bar{\mathfrak{p}})$ of the residue field $\mathrm{H}^{0}(A)_{\bar{\mathfrak{p}}}/\bar{\mathfrak{p}}\mathrm{H}^{0}(A)_{\bar{\mathfrak{p}}}$.

Next we bring several necessary lemmas and their proofs.

\begin{lem}\label{lem0.2} $(1)$ For $X\in \mathrm{D}^{-}_{\mathrm{f}}(A)$ and $Y\in\mathrm{D}^{+}(A)$, one has
\begin{center}$\mathrm{inf}\mathrm{RHom}_A(X,Y)=\mathrm{inf}\{\mathrm{inf}\mathrm{RHom}_{A}(\mathrm{H}^\ell(X),Y)-\ell
\hspace{0.03cm}|\hspace{0.03cm}\ell\in\mathbb{Z}\}$.\end{center}

$(2)$ For $X\in \mathrm{D}^{-}_{\mathrm{f}}(A)$ and $Z\in \mathrm{D}^{-}(A)$, one has
\begin{center}$\mathrm{sup}(X\otimes_A^\mathrm{L}Z)=\mathrm{sup}\{\mathrm{sup}(\mathrm{H}^\ell(X)\otimes_A^\mathrm{L}Z)+\ell
\hspace{0.03cm}|\hspace{0.03cm}\ell\in\mathbb{Z}\}$.\end{center}

$(3)$ Let $(A,\bar{\mathfrak{m}},\bar{\kappa})$ be a local DG-ring. For $X,Y\in\mathrm{D}^-_\mathrm{f}(A)$, one has
\begin{center}$\mathrm{sup}(X\otimes^\mathrm{L}_AY)=\mathrm{sup}X+\mathrm{sup}Y$.\end{center}
\end{lem}
\begin{proof} (1) We first show that $\mathrm{inf}\mathrm{RHom}_A(X,Y)=\mathrm{inf}\{\mathrm{depth}_{A_{\bar{\mathfrak{p}}}}Y_{\bar{\mathfrak{p}}}
-\mathrm{sup}X_{\bar{\mathfrak{p}}}\hspace{0.03cm}|\hspace{0.03cm}\bar{\mathfrak{p}}\in\mathrm{Spec}\mathrm{H}^0(A)\}$. For each $\bar{\mathfrak{p}}\in\mathrm{Spec}\mathrm{H}^0(A)$ one has the next (in)equalities
\begin{center}$\begin{aligned}r:=\mathrm{inf}\mathrm{RHom}_A(X,Y)
&\leq\mathrm{inf}\mathrm{RHom}_{A_{\bar{\mathfrak{p}}}}(X_{\bar{\mathfrak{p}}},Y_{\bar{\mathfrak{p}}})\\
&\leq\mathrm{depth}_{A_{\bar{\mathfrak{p}}}}\mathrm{RHom}_{A_{\bar{\mathfrak{p}}}}(X_{\bar{\mathfrak{p}}},Y_{\bar{\mathfrak{p}}})\\
&=\mathrm{depth}_{A_{\bar{\mathfrak{p}}}}Y_{\bar{\mathfrak{p}}}
-\mathrm{sup}X_{\bar{\mathfrak{p}}},\end{aligned}$\end{center}
where the first inequality is by \cite[Proposition 2.4]{Mi19}, the second inequality is by \cite[Proposition 3.3]{Shaul} and the equality is by \cite[Proposition 4.9]{ya20} and the Nakayama's lemma. If $\bar{\mathfrak{p}}\in\mathrm{Ass}_{\mathrm{H}^0(A)}\mathrm{H}^r(\mathrm{RHom}_A(X,Y))$, then the claimed equality holds by \cite[Proposition 3.3]{Shaul}.

Set $i=\mathrm{inf}\mathrm{RHom}_A(X,Y)$ and $j=\mathrm{inf}\{\mathrm{inf}\mathrm{RHom}_{A}(\mathrm{H}^\ell(X),Y)-\ell
\hspace{0.03cm}|\hspace{0.03cm}\ell\in\mathbb{Z}\}$.
 For $\bar{\mathfrak{p}}\in\mathrm{Spec}\mathrm{H}^0(A)$, set $s=\mathrm{sup}X_{\bar{\mathfrak{p}}}$. By the preceding proof, we have
\begin{center}$\mathrm{depth}_{A_{\bar{\mathfrak{p}}}}Y_{\bar{\mathfrak{p}}}
-\mathrm{sup}X_{\bar{\mathfrak{p}}}=\mathrm{depth}_{A_{\bar{\mathfrak{p}}}}Y_{\bar{\mathfrak{p}}}
-s\geq\mathrm{inf}\mathrm{RHom}_{A}(\mathrm{H}^s(X),Y)-s\geq j$,\end{center}and hence
 $i\geq j$.
On the other hand, let $\ell\in\mathbb{Z}$ and $\mathrm{H}(\mathrm{RHom}_{A}(\mathrm{H}^\ell(X),Y))\not\simeq0$. Choose $\bar{\mathfrak{p}}\in\mathrm{Spec}\mathrm{H}^0(A)$ such that $\mathrm{inf}\mathrm{RHom}_{A}(\mathrm{H}^\ell(X),Y)=\mathrm{depth}_{A_{\bar{\mathfrak{p}}}}Y_{\bar{\mathfrak{p}}}$. Then
\begin{center}$\mathrm{RHom}_{A}(\mathrm{H}^\ell(X),Y)-\ell\geq\mathrm{depth}_{A_{\bar{\mathfrak{p}}}}Y_{\bar{\mathfrak{p}}}
-\mathrm{sup}X_{\bar{\mathfrak{p}}}\geq i$.\end{center}Therefore, the equality holds.

(2) Set $E=\oplus_{\bar{\mathfrak{m}}\in\mathrm{Max}\mathrm{H}^0(A)}E(A,\bar{\mathfrak{m}})$. We have the following equalities
\begin{center}$\begin{aligned}\mathrm{sup}(X\otimes_A^\mathrm{L}Z)
&=-\mathrm{inf}\mathrm{RHom}_{A}(X,\mathrm{RHom}_{A}(Z,E))\\
&=-\mathrm{inf}\{\mathrm{inf}\mathrm{RHom}_{A}(\mathrm{H}^\ell(X),\mathrm{RHom}_{A}(Z,E))-\ell
\hspace{0.03cm}|\hspace{0.03cm}\ell\in\mathbb{Z}\}\\
&=\mathrm{sup}\{\ell-\mathrm{inf}\mathrm{RHom}_{A}(\mathrm{H}^\ell(X)\otimes_A^\mathrm{L}Z,E)
\hspace{0.03cm}|\hspace{0.03cm}\ell\in\mathbb{Z}\}\\
&=\mathrm{sup}\{\ell+\mathrm{sup}(\mathrm{H}^\ell(X)\otimes_A^\mathrm{L}Z)
\hspace{0.03cm}|\hspace{0.03cm}\ell\in\mathbb{Z}\},\end{aligned}$\end{center}
where the first and the last ones are by \cite[Theorem 4.10]{s18}, the second one is by (1).

(3) As $\mathrm{H}^{\mathrm{sup}X}(X)$ and $\mathrm{H}^{\mathrm{sup}Y}(Y)$ are finitely generated $\mathrm{H}^0(A)$-modules,
\begin{center}$\mathrm{H}^{\mathrm{sup}X+\mathrm{sup}X}(\bar{\kappa}\otimes^\mathrm{L}_AX
\otimes^\mathrm{L}_AY)\cong\bar{\kappa}\otimes_{\mathrm{H}^0(A)}\mathrm{H}^{\mathrm{sup}X}(X)\otimes_{\bar{\kappa}} \bar{\kappa}\otimes_{\mathrm{H}^0(A)}\mathrm{H}^{\mathrm{sup}Y}(Y)\neq0$,\end{center}
which implies the desired equality.
\end{proof}

By the equality (2.14) of \cite[Proposition 2.13]{Shaul}, for $X\in\mathrm{D}^{-}(A)$, one has \begin{center}$\mathrm{lc.dim}_AX=\mathrm{sup}\{\mathrm{dim}\mathrm{H}^0(A)/\bar{\mathfrak{p}}+\mathrm{sup}X_{\bar{\mathfrak{p}}}
\hspace{0.03cm}|\hspace{0.03cm}\bar{\mathfrak{p}}\in\mathrm{Spec}\mathrm{H}^0(A)\}$.\end{center}

\begin{lem}\label{lem0.4} For $X\in \mathrm{D}^{-}_{\mathrm{f}}(A)$ and $Y\in \mathrm{D}^{-}(A)$, there exists an equality
\begin{center}$\mathrm{lc.dim}_A(X\otimes_A^\mathrm{L}Y)=\mathrm{sup}\{\mathrm{lc.dim}_A(\mathrm{H}^\ell(X)\otimes_A^\mathrm{L}Y)+\ell
\hspace{0.03cm}|\hspace{0.03cm}\ell\in\mathbb{Z}\}$.\end{center}
\end{lem}
\begin{proof}  We have the following computations
\begin{center}$\begin{aligned}\mathrm{lc.dim}(X\otimes_A^\mathrm{L}Y)
&=\mathrm{sup}\{\mathrm{dim}\mathrm{H}^0(A)/\bar{\mathfrak{p}}+\mathrm{sup}(X\otimes_A^\mathrm{L}Y)_{\bar{\mathfrak{p}}}
\hspace{0.03cm}|\hspace{0.03cm}\bar{\mathfrak{p}}\in\mathrm{Spec}\mathrm{H}^0(A)\}\\
&=\mathrm{sup}\{\mathrm{dim}\mathrm{H}^0(A)/\bar{\mathfrak{p}}+
\mathrm{sup}\{\mathrm{sup}(\mathrm{H}^\ell(X_{\bar{\mathfrak{p}}})\otimes_{A_{\bar{\mathfrak{p}}}}^\mathrm{L}Y_{\bar{\mathfrak{p}}})+\ell
\hspace{0.03cm}|\hspace{0.03cm}\ell\in\mathbb{Z}\}
\hspace{0.03cm}|\hspace{0.03cm}\bar{\mathfrak{p}}\in\mathrm{Spec}\mathrm{H}^0(A)\}\\
&=\mathrm{sup}\{\mathrm{dim}\mathrm{H}^0(A)/\bar{\mathfrak{p}}+
\mathrm{sup}(\mathrm{H}^\ell(X)\otimes_{A}^\mathrm{L}Y)_{\bar{\mathfrak{p}}}+\ell
\hspace{0.03cm}|\hspace{0.03cm}\bar{\mathfrak{p}}\in\mathrm{Spec}\mathrm{H}^0(A),\ell\in\mathbb{Z}\}\\
&=\mathrm{sup}\{\mathrm{lc.dim}_A(\mathrm{H}^\ell(X)\otimes_{A}^\mathrm{L}Y)+\ell
\hspace{0.03cm}|\hspace{0.03cm}\ell\in\mathbb{Z}\},\end{aligned}$\end{center}
where the second equality follows by Lemma \ref{lem0.2}(2), as desired.
\end{proof}

\vspace{1.5mm}
{\bf Dualizing DG-module.} Following \cite{s18}, a dualizing DG-module $R$ over a DG-ring $A$ is a DG-module $R\in\mathrm{D}^{+}_{\mathrm{f}}(A)$ such that $\mathrm{injdim}_AR<\infty$ and the natural map $A\rightarrow\mathrm{RHom}_A(R,R)$ is an isomorphism in $\mathrm{D}(A)$. If $A$ is local then dualizing DG-module over $A$ is unique up to shift. We say that $R$ is normalized if $\mathrm{inf}R=-\mathrm{dim}\mathrm{H}^{0}(A)$. The next result is a DG-version of Grothendieck's local duality, which is a nice generalization of \cite[Theorem 7.26]{s18}.

\begin{thm}\label{lem0.7} Let $(A,\bar{\mathfrak{m}},\bar{\kappa})$ be a local DG-ring with $\mathrm{amp}A<\infty$ and $R$ a normalized dualizing DG-module over $A$.

$(1)$ For every $X\in\mathrm{D}_\mathrm{f}(A)$, one has an isomorphism in $\mathrm{D}(A)$
\begin{center}$\mathrm{R}\Gamma_{\bar{\mathfrak{m}}}(X)\simeq\mathrm{RHom}_A(\mathrm{RHom}_A(X,R),E(A,\bar{\mathfrak{m}}))$. \end{center}

$(2)$ For every $Y\in\mathrm{D}^-_{\mathrm{f}}(A)$, one has an isomorphism in $\mathrm{D}(A)$
 \begin{center}$\mathrm{L}\Lambda^{\bar{\mathfrak{m}}}(Y)\simeq\mathrm{RHom}_A(\mathrm{RHom}_A(Y,E(A,\bar{\mathfrak{m}})),R)$. \end{center}
\end{thm}
\begin{proof} (1) Set $E=E(A,\bar{\mathfrak{m}})$ and $\mathrm{Tel}(A;\bar{\mathfrak{m}})=\mathrm{Tel}(A^0;\mathfrak{m})\otimes_{A^0}A$. First assume that $X\in\mathrm{D}^-_\mathrm{f}(A)$. It follows by Lemma \ref{lem0.15} and \cite[Proposition 7.25]{s18} that
\begin{center}$X\otimes_A^\mathrm{L}\mathrm{Tel}(A,\bar{\mathfrak{m}})\simeq X\otimes_A^\mathrm{L}\mathrm{RHom}_A(R,R\otimes_A^\mathrm{L}\mathrm{Tel}(A,\bar{\mathfrak{m}}))\simeq\mathrm{RHom}_A(\mathrm{RHom}_A(X,R),E)$. \end{center}Now assume that $X\in\mathrm{D}_\mathrm{f}(A)$, set $X'=\cdots\longrightarrow X^{n-2}\xrightarrow{d_{n-2}} X^{n-1}\xrightarrow{d_{n-1}}\mathrm{im}d^{n-1}\longrightarrow0$ and $X''=0\longrightarrow X^n/\mathrm{im}d^{n-1}\xrightarrow{\bar{d}_{n}} X^{n+1}\xrightarrow{d_{n+1}}X^{n+2}\longrightarrow\cdots$. The exact triangle $X'\rightarrow X\rightarrow X''\rightsquigarrow$ induces a commutative diagram of exact triangles in $\mathrm{D}(A)$:
\begin{center}$\xymatrix@C=13pt@R=15pt{
\mathrm{R}\Gamma_{\bar{\mathfrak{m}}}(X') \ar[d]\ar[r]& \mathrm{R}\Gamma_{\bar{\mathfrak{m}}}(X)\ar[d]\ar[r] &\mathrm{R}\Gamma_{\bar{\mathfrak{m}}}(X'')\ar[d]\\
\mathrm{RHom}_A(\mathrm{RHom}_A(X',R),E)\ar[r]& \mathrm{RHom}_A(\mathrm{RHom}_A(X,R),E)\ar[r] &\mathrm{RHom}_A(\mathrm{RHom}_A(X'',R),E)}$
\end{center}As $X'\in\mathrm{D}^-_\mathrm{f}(A)$ and $X''\in\mathrm{D}^+_\mathrm{f}(A)$, it follows by \cite[Theorem 7.26]{s18} and the preceding proof that the first and the third vertical morphisms are isomorphisms in $\mathrm{D}(A)$. Consequently, the middle vertical morphism is an isomorphism in $\mathrm{D}(A)$, as desired.

 (2) Since $Y\simeq\mathrm{RHom}_A(\mathrm{RHom}_A(Y,R),R)$, $\mathrm{L}\Lambda^{\bar{\mathfrak{m}}}(Y)\simeq\mathrm{RHom}_A(\mathrm{R}\Gamma_{\bar{\mathfrak{m}}}(\mathrm{RHom}_A(Y,R)),R)
\simeq\mathrm{RHom}_A(\mathrm{RHom}_A(Y,\mathrm{R}\Gamma_{\bar{\mathfrak{m}}}(R)),R)
\simeq\mathrm{RHom}_A(\mathrm{RHom}_A(Y,E(A,\bar{\mathfrak{m}})),R)$ by Lemma \ref{lem0.15}(2) and \cite[Proposition 7.25]{s18}.
\end{proof}

\begin{cor}\label{lem0.8} Let $(A,\bar{\mathfrak{m}},\bar{\kappa})$ be a local DG-ring with $\mathrm{amp}A<\infty$ and $R$ a dualizing DG-module for $A$. For $X\in\mathrm{D}^-_\mathrm{f}(A)$ and $Y\in\mathrm{D}^+_\mathrm{f}(A)$, one has
\begin{center}$\mathrm{inf}\mathrm{RHom}_A(X,R)=\mathrm{depth}_AR-\mathrm{lc.dim}_AX$, $\mathrm{depth}_A\mathrm{RHom}_A(X,R)=\mathrm{depth}_AR-\mathrm{sup}X$,\end{center}
\begin{center}$\mathrm{sup}\mathrm{RHom}_A(Y,R)=\mathrm{depth}_AR-\mathrm{depth}_AY$, $\mathrm{lc.dim}_A\mathrm{RHom}_A(Y,R)=\mathrm{depth}_AR-\mathrm{inf}Y$,\end{center}
\begin{center}$\mathrm{injdim}_A\mathrm{RHom}_A(X,R)=\mathrm{projdim}_AX+\mathrm{depth}_AR$, $\mathrm{projdim}_A\mathrm{RHom}_A(Y,R)=\mathrm{injdim}_AY-\mathrm{depth}_AR$.\end{center}In particular, if $R$ is normalized then $\mathrm{depth}A=-\mathrm{sup}R$ and
$\mathrm{injdim}_{A}R=0=\mathrm{depth}_{A}R$.
\end{cor}
\begin{proof} By Lemma \ref{lem0.15}(1), we have \begin{center}$\bar{\kappa}\otimes^\mathrm{L}_{A}\mathrm{RHom}_A(R,R)\simeq\mathrm{RHom}_A(\mathrm{RHom}_A(\bar{\kappa},R),R)\simeq
\mathrm{RHom}_{\bar{\kappa}}(\mathrm{RHom}_A(\bar{\kappa},R),\mathrm{RHom}_A(\bar{\kappa},R))$,\end{center} it follows by \cite[Corollary 2.3]{Mi19} that $0=-\mathrm{inf}(\bar{\kappa}\otimes^\mathrm{L}_{A}\mathrm{RHom}_A(R,R))=-\mathrm{inf}\mathrm{RHom}_A(\bar{\kappa},R)+\mathrm{sup}\mathrm{RHom}_A(\bar{\kappa},R)=
-\mathrm{depth}_AR+\mathrm{injdim}_AR$.
Set $R'=R[-\mathrm{inf}R-\mathrm{lc.dim}A]$. Then $R'$ is a normalized dualizing DG-module over $A$.
  Hence we have the following equalities
\begin{center}$\begin{aligned}\mathrm{inf}\mathrm{RHom}_A(X,R)
&=\mathrm{inf}\mathrm{RHom}_A(X,R')+\mathrm{inf}R+\mathrm{lc.dim}A\\
&=-\mathrm{sup}\mathrm{RHom}_A(\mathrm{RHom}_A(X,R'),E(A,\bar{\mathfrak{m}}))+\mathrm{inf}R+\mathrm{lc.dim}A\\
&=-\mathrm{sup}\mathrm{R}\Gamma_{\bar{\mathfrak{m}}}(X)+\mathrm{inf}R+\mathrm{lc.dim}A\\
&=-\mathrm{lc.dim}_AX+\mathrm{injdim}_AR\\
&=\mathrm{depth}_AR-\mathrm{lc.dim}_AX,\end{aligned}$\end{center}
where the second one is by \cite[Theorem 4.10]{s18}, the third one is by Theorem \ref{lem0.7}, the fourth one is by \cite[Lemma 2.7]{BSSW}.
We also have the following equalities
\begin{center}$\begin{aligned}\mathrm{sup}\mathrm{RHom}_A(Y,R)
&=\mathrm{sup}\mathrm{RHom}_A(Y,R')+\mathrm{inf}R+\mathrm{lc.dim}A\\
&=-\mathrm{inf}\mathrm{RHom}_A(\mathrm{RHom}_A(Y,R'),E(A,\bar{\mathfrak{m}}))+\mathrm{inf}R+\mathrm{lc.dim}A\\
&=-\mathrm{inf}\mathrm{R}\Gamma_{\bar{\mathfrak{m}}}(Y)+\mathrm{injdim}_AR\\
&=\mathrm{depth}_AR-\mathrm{depth}_AY,\end{aligned}$\end{center}
where the second one is by \cite[Theorem 4.10]{s18}, the third one is by Theorem \ref{lem0.7}.
Therefore, $\mathrm{lc.dim}_A\mathrm{RHom}_A(Y,R)=\mathrm{depth}_AR-\mathrm{inf}Y$ and $\mathrm{depth}_A\mathrm{RHom}_A(X,R)=\mathrm{depth}_AR-\mathrm{sup}X$.

By \cite[Corollary 2.31]{Mi19}, one has the following equalities:\begin{center}$\begin{aligned}\mathrm{injdim}_A\mathrm{RHom}_A(X,R)
&=\mathrm{sup}\mathrm{RHom}_A(\bar{\kappa},\mathrm{RHom}_A(X,R))\\
&=\mathrm{sup}\mathrm{RHom}_{\bar{\kappa}}(\bar{\kappa}\otimes^\mathrm{L}_AX,\mathrm{RHom}_A(\bar{\kappa},R))\\
&=\mathrm{sup}\mathrm{RHom}_A(\bar{\kappa},R)-\mathrm{inf}(\bar{\kappa}\otimes^\mathrm{L}_AX)\\
&=\mathrm{projdim}_AX+\mathrm{depth}_AR.\end{aligned}$\end{center}
By \cite[Corollary 2.3]{Mi19} and Lemma \ref{lem0.15}(1), one has the following equalities:\begin{center}$\begin{aligned}\mathrm{projdim}_A\mathrm{RHom}_A(Y,R)
&=-\mathrm{inf}(\bar{\kappa}\otimes^\mathrm{L}_A\mathrm{RHom}_A(Y,R))\\
&=-\mathrm{inf}\mathrm{RHom}_{\bar{\kappa}}(\mathrm{RHom}_A(\bar{\kappa},Y),\mathrm{RHom}_A(\bar{\kappa},R))\\
&=-\mathrm{inf}\mathrm{RHom}_A(\bar{\kappa},R)+\mathrm{sup}\mathrm{RHom}_A(\bar{\kappa},Y)\\
&=\mathrm{injdim}_AY-\mathrm{depth}_AR.\end{aligned}$\end{center}We obtain the equalities we seek.
\end{proof}

\section{\bf Intersection Theorem}
In this section, we examine the Intersection Theorem (IT) for DG-modules and study the local cohomology Krull dimension of the DG-modules $\mathrm{RHom}_A(X,Y)$ and $X\otimes_A^\mathrm{L}Y$, which extends some classical results proved by Foxby and Christensen in \cite{F,CF}.

We begin with the following main theorem of this section, which is an extension of IT in \cite[18.7]{F} and \cite[Theorem 2.6]{DY1} in DG-setting.

\begin{thm}\label{lem5.4} Let $(A,\bar{\mathfrak{m}},\bar{\kappa})$ be a local DG-ring with $\mathrm{amp}A<\infty$.

$(1)$ Let $X\in\mathrm{D}^-(A)$ with $\mathrm{flatdim}_AX<\infty$ and
$0\not\simeq Y\in\mathrm{D}^{-}_{\mathrm{f}}(A)$.
If $\mathrm{depth}_AX<\infty$, then
 \begin{center}$\mathrm{lc.dim}_AY\leq\mathrm{lc.dim}_A(X\otimes^\mathrm{L}_AY)+\mathrm{flatdim}_AX$.\end{center}

$(2)$ For $0\not\simeq X,Y\in\mathrm{D}^{-}_{\mathrm{f}}(A)$ with $\mathrm{projdim}_AX<\infty$, one has an inequality
 \begin{center}$\mathrm{lc.dim}_AY\leq\mathrm{lc.dim}_A(X\otimes^\mathrm{L}_AY)+\mathrm{projdim}_AX$.\end{center}
\end{thm}
\begin{proof} (1) A straightforward
computation based on
 Lemma \ref{lem0.4}:
\begin{center}$\begin{aligned}\mathrm{lc.dim}_A(X\otimes^\mathrm{L}_AY)
&=\mathrm{sup}\{\mathrm{lc.dim}_{A}(X\otimes^\mathrm{L}_A\mathrm{H}^\ell(Y))+\ell|\ell\in\mathbb{Z}\}\\
&=\mathrm{sup}\{\mathrm{lc.dim}_{\mathrm{H}^0(A)}(X\otimes^\mathrm{L}_A\mathrm{H}^0(A)\otimes^\mathrm{L}_{\mathrm{H}^0(A)}\mathrm{H}^\ell(Y))+\ell|\ell\in\mathbb{Z}\}\\
&\geq\mathrm{sup}\{\mathrm{dim}_{\mathrm{H}^0(A)}\mathrm{H}^\ell(Y))+\ell|\ell\in\mathbb{Z}\}-
\mathrm{flatdim}_{\mathrm{H}^0(A)}(X\otimes^\mathrm{L}_A\mathrm{H}^0(A))\\
&=\mathrm{lc.dim}_AY-\mathrm{flatdim}_AX,\end{aligned}$\end{center}  where the
 inequality follows by \cite[Theorem 2.3]{IMSW} and \cite[19.7]{F}, the last equality follows from \cite[Corollary 1.5]{BSSW}.

(2) As $\mathrm{projdim}_AX<\infty$, it follows by \cite[Theorem 4.5]{ya20} that $\mathrm{depth}_AX<\infty$.
Hence (1) yields the desired inequality.
\end{proof}

\begin{cor}\label{lem1.1'} Let $(A,\bar{\mathfrak{m}},\bar{\kappa})$ be a local DG-ring with $\mathrm{amp}A<\infty$ and $Y\in\mathrm{D}^{-}(A)$ with $\mathrm{depth}_AY<\infty$. If either $0\not\simeq X\in\mathrm{D}^{-}_{\mathrm{f}}(A)$ or $\mathrm{depth}_AX<\infty$, then there is an inequality, \begin{center}$\mathrm{flatdim}_AY\geq\mathrm{depth}_AX-\mathrm{lc.dim}_A(Y\otimes_A^\mathrm{L}X)$.\end{center}
\end{cor}
\begin{proof}  One can assume that $\mathrm{flatdim}_AY<\infty$.

If $0\not\simeq X\in\mathrm{D}^{-}_{\mathrm{f}}(A)$, then we have the next inequalities\begin{center}$\begin{aligned}\mathrm{flatdim}_AY
&\geq\mathrm{lc.dim}_AX-\mathrm{lc.dim}_A(Y\otimes_A^\mathrm{L}X)\\
&\geq\mathrm{supR}\Gamma_{\bar{\mathfrak{m}}}(X)-\mathrm{lc.dim}_A(Y\otimes_A^\mathrm{L}X)\\
&\geq\mathrm{infR}\Gamma_{\bar{\mathfrak{m}}}(X)-\mathrm{lc.dim}_A(Y\otimes_A^\mathrm{L}X)\\
&=\mathrm{depth}_AX-\mathrm{lc.dim}_A(Y\otimes_A^\mathrm{L}X),\end{aligned}$\end{center} where the first one is by Theorem \ref{lem5.4}(1), the second one is by \cite[Theorem 2.15]{Shaul}, the third one is by $\mathrm{R}\Gamma_{\bar{\mathfrak{m}}}(X)\not\simeq0$ and the equality is by \cite[Proposition 3.3]{Shaul}.

If $\mathrm{depth}_AX<\infty$, then $\mathrm{depth}_A(Y\otimes_A^\mathrm{L}X)<\infty$ and so $\mathrm{depth}_A(Y\otimes_A^\mathrm{L}X)\leq\mathrm{lc.dim}_A(Y\otimes_A^\mathrm{L}X)$. By \cite[Theorem 4.8]{ya20},
one has $\mathrm{flatdim}_AY
=\mathrm{depth}_AX-\mathrm{depth}_A(Y\otimes_A^\mathrm{L}X)$, which implies the desired inequality.
\end{proof}

\begin{cor}\label{lem1.1} Let $(A,\bar{\mathfrak{m}},\bar{\kappa})$ be a local DG-ring with $\mathrm{amp}A<\infty$, and let $X\in\mathrm{D}^{-}(A)$ with $\mathrm{flatdim}_AX<\infty$ and $0\not\simeq Y\in\mathrm{D}^{-}_{\mathrm{f}}(A)$.
If either $\mathrm{depth}_AX<\infty$ or $0\not\simeq X\in\mathrm{D}_{\mathrm{f}}(A)$, then there exists an inequality, \begin{center}$\mathrm{lc.dim}_A(X\otimes_A^\mathrm{L}Y)-\mathrm{depth}_A(X\otimes_A^\mathrm{L}Y)
\geq\mathrm{lc.dim}_AY-\mathrm{depth}_AY$.\end{center}In particular, $\mathrm{lc.dim}_AX-\mathrm{depth}_AX\geq\mathrm{lc.dim}A-\mathrm{depth}A$.
\end{cor}
\begin{proof} By \cite[Corollary 2.31]{Mi19} and \cite[Proposition 6.14]{BSSW}, one has $\mathrm{flatdim}_AX=-\mathrm{inf}(\bar{\kappa}\otimes^\mathrm{L}_AX)$. If $X\in\mathrm{D}^{-}_{\mathrm{f}}(A)$, then $\mathrm{projdim}_AX=\mathrm{flatdim}_AX$ by \cite[Proposition 2.2]{Mi19}.
As $\mathrm{R}\Gamma_{\bar{\mathfrak{m}}}(\bar{\kappa})\simeq\bar{\kappa}$, it follows by \cite[Theorem 4.8]{ya20} that $\mathrm{inf}(\bar{\kappa}\otimes^\mathrm{L}_AX)=\mathrm{inf}\mathrm{R}\Gamma_{\bar{\mathfrak{m}}}(\bar{\kappa}\otimes^\mathrm{L}_AX)=
\mathrm{depth}_A(\bar{\kappa}\otimes^\mathrm{L}_AX)=\mathrm{depth}_AX-\mathrm{depth}A$. By Theorem \ref{lem5.4} and \cite[Theorem 4.8]{ya20}, one has
\begin{center}$\mathrm{lc.dim}_A(X\otimes_A^\mathrm{L}Y)-\mathrm{depth}_A(X\otimes_A^\mathrm{L}Y)
\geq\mathrm{lc.dim}_AY-\mathrm{flatdim}_AX+\mathrm{depth}A-\mathrm{depth}_AX-\mathrm{depth}_AY$.\end{center}
Thus we obtain the desired inequality.
\end{proof}

Following \cite{Shaul},
a local DG-ring $(A,\bar{\mathfrak{m}})$ with $\mathrm{amp}A<\infty$ is called \emph{local Cohen-Macaulay} if $\mathrm{amp}\mathrm{R}\Gamma_{\bar{\mathfrak{m}}}(A)=\mathrm{amp}A$. A DG-module $X\in\mathrm{D}^{\mathrm{b}}_{\mathrm{f}}(A)$ is called \emph{local Cohen-Macaulay} if
\begin{center}$\mathrm{amp}X=\mathrm{amp}A=\mathrm{amp}\mathrm{R}\Gamma_{\bar{\mathfrak{m}}}(X)$.\end{center}

\begin{rem}\label{lem:2.11}{\rm Let $(A,\bar{\mathfrak{m}},\bar{\kappa})$ be a local DG-ring with $\mathrm{amp}A<\infty$.

(1) The next example shows that the condition $\mathrm{depth}_AX<\infty$ or $X\in\mathrm{D}^{-}_{\mathrm{f}}(A)$ can not be dropped in Corollaries \ref{lem1.1'} and \ref{lem1.1}.

Set $(A,\bar{\mathfrak{m}})=(R,\mathfrak{m})$ be a local ring, $\textbf{\emph{x}}=x_1,\cdots,x_n$ an $R$-regular sequence in $\mathfrak{m}$ and $\mathfrak{p}$ a prime ideal such that $\textbf{\emph{x}}\not\in\mathfrak{p}$. Set $X=E(R/\mathfrak{p})$ and $Y=R/\textbf{\emph{x}}$. Then $X\not\in\mathrm{D}^{-}_{\mathrm{f}}(R)$ and $\mathrm{depth}_RX=\infty$. Thus $\mathrm{flatdim}_RY=n<\mathrm{depth}_RX-\mathrm{lc.dim}_R(Y\otimes_R^\mathrm{L}X)$ and $\mathrm{lc.dim}_RX-\mathrm{depth}_RX=-\infty<\mathrm{lc.dim}R-\mathrm{depth}R$.

 (2) Let $X\in\mathrm{D}^-(A)$ with $\mathrm{depth}_AX<\infty$.  One can assume that $\mathrm{flatdim}_AX<\infty$. Then $\mathrm{flatdim}_AX=-\mathrm{depth}_A(\bar{\kappa}\otimes^\mathrm{L}_AX)=\mathrm{depth}A-\mathrm{depth}_AX$, it follows by Corollary \ref{lem1.1} that $\mathrm{flatdim}_AX\geq\mathrm{lc.dim}A-\mathrm{lc.dim}_AX$. If $X\in\mathrm{D}^{-}_{\mathrm{f}}(A)$, then $\mathrm{depth}_AX<\infty$ by \cite[Proposition 3.5]{Shaul}, and so $\mathrm{projdim}_AX\geq\mathrm{lc.dim}A-\mathrm{lc.dim}_AX$.

 (3) Let $0\not\simeq X\in\mathrm{D}^-(A)$ with $\mathrm{flatdim}_AX<\infty$. If $X\simeq\mathrm{R}\Gamma_{\bar{\mathfrak{m}}}(X)$, then $\mathrm{lc.dim}_AX=\mathrm{sup}X$ and $\mathrm{depth}_AX=\mathrm{inf}X$, and hence $\mathrm{amp}X\geq\mathrm{lc.dim}A-\mathrm{depth}A$.

(4) If there is a local Cohen-Macaulay DG-module
$X\in\mathrm{D}^{\mathrm{b}}_{\mathrm{f}}(A)$ with $\mathrm{projdim}_AX<\infty$, then $\mathrm{amp}\mathrm{R}\Gamma_{\bar{\mathfrak{m}}}(X)=\mathrm{lc.dim}_AX-\mathrm{depth}_AX
\geq\mathrm{lc.dim}A-\mathrm{depth}A=\mathrm{amp}\mathrm{R}\Gamma_{\bar{\mathfrak{m}}}(A)$ by Corollary \ref{lem1.1}. As
$\mathrm{amp}\mathrm{R}\Gamma_{\bar{\mathfrak{m}}}(A)\geq\mathrm{amp}A$ by \cite[Theorem 4.1(1)]{Shaul}, it implies that $\mathrm{amp}\mathrm{R}\Gamma_{\bar{\mathfrak{m}}}(A)=\mathrm{amp}A$ and $A$ is local Cohen-Macaulay.

(5) For $0\not\simeq X,Y\in\mathrm{D}^{\mathrm{b}}_{\mathrm{f}}(A)$ with $\mathrm{projdim}_AX<\infty$, if $\mathrm{Supp}_AX\cap\mathrm{Supp}_AY=\{\bar{\mathfrak{m}}\}$, then $\mathrm{lc.dim}_A(X\otimes_A^\mathrm{L}Y)=\mathrm{sup}\{\mathrm{dim}\mathrm{H}^{0}(A)/\bar{\mathfrak{p}}+\mathrm{sup}(X\otimes_A^\mathrm{L}Y)_{\bar{\mathfrak{p}}}
\hspace{0.03cm}|\hspace{0.03cm}\bar{\mathfrak{p}}\in\mathrm{Supp}_AX\cap\mathrm{Supp}_AY\}
=\mathrm{sup}X+\mathrm{sup}Y$. Thus $\mathrm{lc.dim}_AY-\mathrm{sup}Y\leq\mathrm{projdim}_AX+\mathrm{sup}X$ by Theorem \ref{lem5.4}(2).}
\end{rem}

The next result is a generalization of \cite[Theorem 3.7]{Ya} and \cite[Theorems 2.1 and 2.6]{DY1}.

\begin{thm}\label{lem5.4'} Let $(A,\bar{\mathfrak{m}},\bar{\kappa})$ be a local DG-ring with $\mathrm{amp}A<\infty$, and let $X,Y\in \mathrm{D}(A)$ such that $\mathrm{RHom}_A(X,Y)\not\simeq0$.

$(i)$ If $X\in \mathrm{D}^{-}_{\mathrm{f}}(A)$ with $\mathrm{projdim}_AX<\infty$ and $Y\in \mathrm{D}^{\mathrm{b}}(A)$, then there is an inequality \begin{center}$\mathrm{lc.dim}_{A}\mathrm{RHom}_A(X,Y)\leq\mathrm{lc.dim}_{A}Y+\mathrm{projdim}_AX$.\end{center}

 $(ii)$ If $X,Y\in \mathrm{D}^{\mathrm{b}}_{\mathrm{f}}(A)$ with $\mathrm{projdim}_AX<\infty$, then there are inequalities \begin{center}$\mathrm{lc.dim}_{A}Y-\mathrm{sup}X\leq\mathrm{lc.dim}_{A}\mathrm{RHom}_A(X,Y)
 \leq\mathrm{lc.dim}_{A}(X\otimes_A^\mathrm{L}Y)+\mathrm{projdim}_AX-\mathrm{inf}X$.\end{center}

 $(iii)$ If $X,Y\in \mathrm{D}^{\mathrm{b}}_{\mathrm{f}}(A)$ with either $\mathrm{projdim}_AX<\infty$ or $\mathrm{injdim}_AY<\infty$, then
\begin{center}$\mathrm{lc.dim}_{A}\mathrm{RHom}_A(X,Y)\leq\mathrm{lc.dim}A-\mathrm{inf}X+\mathrm{sup}Y$.\end{center}
\end{thm}
\begin{proof} (i) By \cite[Remark 2.3]{Shaul}, $\mathrm{lc.dim}_{A}\mathrm{RHom}_A(X,Y)\leq\mathrm{dim}\mathrm{H}^{0}(A)+\mathrm{sup}\mathrm{RHom}_A(X,Y)<\infty$.
 One can choose $\bar{\mathfrak{p}}\in\mathrm{Supp}_AX\cap\mathrm{Supp}_AY$ such that \begin{center}$\mathrm{lc.dim}_{A}\mathrm{RHom}_A(X,Y)
=\mathrm{dim}\mathrm{H}^{0}(A)/\bar{\mathfrak{p}}+\mathrm{sup}\mathrm{RHom}_A(X,Y)_{\bar{\mathfrak{p}}}$.\end{center} By Lemmas \ref{lem0.11}(1) and \ref{lem0.15}(2), we have
$\mathrm{RHom}_A(X,Y)_{\bar{\mathfrak{p}}}\simeq\mathrm{RHom}_{A_{\bar{\mathfrak{p}}}}(X_{\bar{\mathfrak{p}}},A_{\bar{\mathfrak{p}}})
\otimes^\mathrm{L}_{A_{\bar{\mathfrak{p}}}}Y_{\bar{\mathfrak{p}}}$
and $\mathrm{sup}\mathrm{RHom}_{A_{\bar{\mathfrak{p}}}}(X_{\bar{\mathfrak{p}}},A_{\bar{\mathfrak{p}}})
\leq\mathrm{sup}\mathrm{RHom}_{A}(X,A)=\mathrm{projdim}_AX$.
Hence
 \begin{center}$\begin{aligned}\mathrm{dim}\mathrm{H}^{0}(A)/\bar{\mathfrak{p}}+\mathrm{sup}\mathrm{RHom}_A(X,Y)_{\bar{\mathfrak{p}}}
&\leq\mathrm{dim}\mathrm{H}^{0}(A)/\bar{\mathfrak{p}}+\mathrm{sup}Y_{\bar{\mathfrak{p}}}+\mathrm{sup}\mathrm{RHom}_{A_{\bar{\mathfrak{p}}}}(X_{\bar{\mathfrak{p}}},A_{\bar{\mathfrak{p}}})\\
&\leq\mathrm{lc.dim}_AY+\mathrm{projdim}_AX.\end{aligned}$\end{center}

(ii) By the proof of (i), one has $\mathrm{lc.dim}_{A}\mathrm{RHom}_A(X,Y)\leq\mathrm{dim}\mathrm{H}^{0}(A)/\bar{\mathfrak{p}}+\mathrm{sup}Y_{\bar{\mathfrak{p}}}+\mathrm{projdim}_AX$ for some $\bar{\mathfrak{p}}\in\mathrm{Supp}_AX\cap\mathrm{Supp}_AY$. By Lemma \ref{lem0.2}(3), we have
\begin{center}$\begin{aligned}\mathrm{lc.dim}_{A}\mathrm{RHom}_A(X,Y)
&\leq\mathrm{dim}\mathrm{H}^{0}(A)/\bar{\mathfrak{p}}+\mathrm{sup}(X\otimes^\mathrm{L}_AY)_{\bar{\mathfrak{p}}}-\mathrm{sup}X_{\bar{\mathfrak{p}}}+\mathrm{projdim}_AX\\
&\leq\mathrm{lc.dim}_A(X\otimes^\mathrm{L}_AY)+\mathrm{projdim}_AX-\mathrm{inf}X.\end{aligned}$\end{center} For the left-hand side of (ii), as $\mathrm{RHom}_A(X,A)\in\mathrm{D}^{\mathrm{b}}_{\mathrm{f}}(A)$ and $\mathrm{projdim}_A\mathrm{RHom}_A(X,A)<\infty$, it follows by Theorem \ref{lem5.4}(2) that
\begin{center}$\begin{aligned}\mathrm{lc.dim}_AY
&\leq\mathrm{lc.dim}_{A}(\mathrm{RHom}_A(X,A)\otimes_A^\mathrm{L}Y)+\mathrm{projdim}_A\mathrm{RHom}_A(X,A)\\
&=\mathrm{lc.dim}_{A}\mathrm{RHom}_A(X,Y)+\mathrm{sup}\mathrm{RHom}_A(\mathrm{RHom}_A(X,A),A)\\
&=\mathrm{lc.dim}_{A}\mathrm{RHom}_A(X,Y)+\mathrm{sup}X,\end{aligned}$\end{center}
where the first equality is by \cite[Proposition 4.4(3)]{ya20} and the second one follows from the isomorphism $X\simeq\mathrm{RHom}_A(\mathrm{RHom}_A(X,A),A)$, as claimed.

(iii) Fix a point $\bar{\mathfrak{p}}\in\mathrm{Spec}\mathrm{H}^0(A)$.
Assume that $\mathrm{projdim}_AX<\infty$.
Then the following computations hold:
\begin{center}$\begin{aligned}\mathrm{dim}\mathrm{H}^{0}(A)/\bar{\mathfrak{p}}+\mathrm{sup}\mathrm{RHom}_{A_{\bar{\mathfrak{p}}}}(X_{\bar{\mathfrak{p}}},Y_{\bar{\mathfrak{p}}})
&\leq\mathrm{dim}\mathrm{H}^{0}(A)/\bar{\mathfrak{p}}+\mathrm{projdim}_{A_{\bar{\mathfrak{p}}}}X_{\bar{\mathfrak{p}}}+\mathrm{sup}Y_{\bar{\mathfrak{p}}}\\
&=\mathrm{dim}\mathrm{H}^{0}(A)/\bar{\mathfrak{p}}+\mathrm{depth}A_{\bar{\mathfrak{p}}}-\mathrm{depth}_{A_{\bar{\mathfrak{p}}}}X_{\bar{\mathfrak{p}}}+\mathrm{sup}Y_{\bar{\mathfrak{p}}}\\
&\leq\mathrm{dim}\mathrm{H}^{0}(A)/\bar{\mathfrak{p}}+\mathrm{dim}\mathrm{H}^0(A)_{\bar{\mathfrak{p}}}+\mathrm{inf}A_{\bar{\mathfrak{p}}}-\mathrm{depth}_{A_{\bar{\mathfrak{p}}}}X_{\bar{\mathfrak{p}}}+\mathrm{sup}Y_{\bar{\mathfrak{p}}}\\
&\leq\mathrm{dim}\mathrm{H}^{0}(A)-\mathrm{inf}X_{\bar{\mathfrak{p}}}+\mathrm{sup}Y_{\bar{\mathfrak{p}}}\\
&\leq\mathrm{lc.dim}A-\mathrm{inf}X+\mathrm{sup}Y,\end{aligned}$\end{center}
where the equality is by \cite[Theorem 4.5]{ya20}, the second inequality is by \cite[Corollary 3.6]{Shaul} and the third one is by \cite[Proposition 3.3]{Shaul}.

Assume that $\mathrm{injdim}_AY<\infty$.
Then the next computations hold:
\begin{center}$\begin{aligned}\mathrm{dim}\mathrm{H}^{0}(A)/\bar{\mathfrak{p}}+\mathrm{sup}\mathrm{RHom}_{A_{\bar{\mathfrak{p}}}}(X_{\bar{\mathfrak{p}}},Y_{\bar{\mathfrak{p}}})
&\leq\mathrm{dim}\mathrm{H}^{0}(A)/\bar{\mathfrak{p}}+\mathrm{injdim}_{A_{\bar{\mathfrak{p}}}}Y_{\bar{\mathfrak{p}}}-\mathrm{inf}X_{\bar{\mathfrak{p}}}\\
&=\mathrm{dim}\mathrm{H}^{0}(A)/\bar{\mathfrak{p}}+\mathrm{depth}A_{\bar{\mathfrak{p}}}+\mathrm{sup}Y_{\bar{\mathfrak{p}}}-\mathrm{inf}X_{\bar{\mathfrak{p}}}\\
&\leq\mathrm{dim}\mathrm{H}^{0}(A)/\bar{\mathfrak{p}}+\mathrm{dim}\mathrm{H}^0(A)_{\bar{\mathfrak{p}}}+\mathrm{inf}A_{\bar{\mathfrak{p}}}+\mathrm{sup}Y_{\bar{\mathfrak{p}}}-\mathrm{inf}X_{\bar{\mathfrak{p}}}\\
&\leq\mathrm{dim}\mathrm{H}^{0}(A)+\mathrm{sup}Y_{\bar{\mathfrak{p}}}-\mathrm{inf}X_{\bar{\mathfrak{p}}}\\
&\leq\mathrm{lc.dim}A+\mathrm{sup}Y-\mathrm{inf}X,\end{aligned}$\end{center}
where the equality is by \cite[Theorem 4.7]{ya20}.

Therefore, we obtain the desired inequality.
\end{proof}

Following \cite{Shaul},
a local DG-ring $(A,\bar{\mathfrak{m}},\bar{\kappa})$ with $\mathrm{amp}A<\infty$ is called \emph{Gorenstein} if $\mathrm{injdim}_AA<\infty$.
The following characterization of Gorenstein DG-rings is a DG-version of Auslander-Bridger characterization of classical Gorenstein rings.

\begin{cor}\label{lem5.2} Let $(A,\bar{\mathfrak{m}},\bar{\kappa})$ be a local DG-ring with $\mathrm{amp}A<\infty$.
 Then the following are equivalent:

 $(1)$ $A$ is local Gorenstein;

 $(2)$ $\mathrm{lc.dim}_{A}\mathrm{RHom}_A(X,A)<\infty$ for all $X\in\mathrm{D}^{\mathrm{b}}_{\mathrm{f}}(A)$;

 $(3)$ $\mathrm{lc.dim}_{A}\mathrm{RHom}_A(T,A)<\infty$ for all finitely generated $\mathrm{H}^{0}(A)$-modules $T$;

 $(4)$ $\mathrm{lc.dim}_{A}\mathrm{RHom}_A(\bar{\kappa},A)<\infty$.
\end{cor}
\begin{proof} (1) $\Rightarrow$ (2) follows from Theorem \ref{lem5.4'}(iii) and (2) $\Rightarrow$ (3) $\Rightarrow$ (4) are trivial.

(4) $\Rightarrow$ (1) One has the following equalities \begin{center}$\mathrm{sup}\mathrm{RHom}_A(\bar{\kappa},A)=
\mathrm{sup}\mathrm{RHom}_A(\bar{\kappa},A)_{\bar{\mathfrak{m}}}=\mathrm{lc.dim}_{A}\mathrm{RHom}_A(\bar{\kappa},A)<\infty$.\end{center} Hence \cite[Corollary 2.31]{Mi19} implies that $\mathrm{injdim}_AA<\infty$ and $A$ is local Gorenstein.
\end{proof}

Let $(A,\bar{\mathfrak{m}},\bar{\kappa})$ be a local DG-ring and $X\in\mathrm{D}^+(A)$. An element $\bar{x}\in\bar{\mathfrak{m}}$ is called \emph{$X$-regular} if it is
$\mathrm{H}^{\mathrm{inf}X}(X)$-regular, that is, if the multiplication map
\begin{center}$\bar{x}:\mathrm{H}^{\mathrm{inf}X}(X)\rightarrow\mathrm{H}^{\mathrm{inf}X}(X)$\end{center}is injective. Inductively, we say that a sequence $\bar{x}_1,\cdots,\bar{x}_n\in\bar{\mathfrak{m}}$ is \emph{$X$-regular} if
$\bar{x}_1$ is $X$-regular and the sequence $\bar{x}_2,\cdots,\bar{x}_n$ is $X/\hspace{-0.15cm}/\bar{x}_1X$-regular. Here, $X/\hspace{-0.15cm}/\bar{x}_1X$ is the
cone of the map $\bar{x}_1:X\rightarrow X$ in $\mathrm{D}(A)$. Following \cite{sh21}, a local DG-ring $(A,\bar{\mathfrak{m}},\bar{\kappa})$ with $\mathrm{amp}A<\infty$ is called  \emph{sequence-regular} if $\bar{\mathfrak{m}}$ is generated by an
$A$-regular sequence. If $A$ is a regular local ring and $M,N$ two finitely generated $A$-modules, then there is an inequality \begin{center}$\mathrm{dim}_{A}M+\mathrm{dim}_{A}N\leq\mathrm{dim}A+\mathrm{dim}_{A}(M\otimes_AN)$,\end{center} which is due to Serre and is called strong Intersection Theorem (see \cite[20.4]{F}). We next provide the DG-version of this theorem.

\begin{prop}\label{lem5.7} Let $(A,\bar{\mathfrak{m}},\bar{\kappa})$ be a sequence-regular DG-ring and $X,Y\in\mathrm{D}^{\mathrm{b}}_{\mathrm{f}}(A)$. Then there exists an inequality
\begin{center}$\mathrm{lc.dim}_{A}X+\mathrm{lc.dim}_{A}Y\leq\mathrm{lc.dim}A+\mathrm{lc.dim}_{A}(X\otimes_A^\mathrm{L}Y)$.\end{center}
\end{prop}
\begin{proof} One has the following (in)equalities
\begin{center}$\begin{aligned}\mathrm{lc.dim}_{A}(X\otimes_A^\mathrm{L}Y)
&=\mathrm{sup}_{s,t\in\mathbb{Z}}\{\mathrm{lc.dim}_{A}(\mathrm{H}^s(X)\otimes^\mathrm{L}_{A}\mathrm{H}^t(Y))+s+t\}\\
&=\mathrm{sup}_{s,t\in\mathbb{Z}}\{\mathrm{dim}_{\mathrm{H}^{0}(A)}(\mathrm{H}^s(X)\otimes_{\mathrm{H}^{0}(A)}\mathrm{H}^t(Y))+s+t\}\\
&\geq\mathrm{sup}_{s,t\in\mathbb{Z}}\{\mathrm{dim}_{\mathrm{H}^{0}(A)}\mathrm{H}^s(X)+s+
\mathrm{dim}_{\mathrm{H}^{0}(A)}\mathrm{H}^t(Y)+t\}-\mathrm{dim}\mathrm{H}^{0}(A)\\
&=\mathrm{lc.dim}_{A}X+\mathrm{lc.dim}_{A}Y-\mathrm{lc.dim}A,\end{aligned}$\end{center}
where the first equality is by Lemma \ref{lem0.4}, the second one is by the equality of the above of Lemma \ref{lem0.4}, the inequality is by \cite[20.4]{F} as $\mathrm{H}^{0}(A)$ is regular by \cite[Theorem 3.4]{sh21}.
\end{proof}

\section{\bf Applications}
In this section, we explore some implications of our previous results. More precisely, the amplitude inequality of DG-modules is improved; the DG-version of NIT, Krull's principle ideal theorem, Bass Conjecture about Cohen-Macaulay rings and Vasconcelos Conjecture about Gorenstein rings are proved; the Minamoto's conjecture is solved completely.

The next proposition improves the J${\o}$rgensen's results, that is to say, eliminate the hypothesis on local DG-rings in \cite[Theorem 3.1]{J}.

\begin{prop}\label{lem5.10} Let $(A,\bar{\mathfrak{m}},\bar{\kappa})$ be a local DG-ring with $\mathrm{amp}A<\infty$ and $0\not\simeq X,Y\in\mathrm{D}^{\mathrm{b}}_{\mathrm{f}}(A)$ with $\mathrm{projdim}_AX<\infty$. One has an inequality \begin{center}$\mathrm{amp}(X\otimes^\mathrm{L}_AY)\geq\mathrm{amp}Y$.\end{center}In particular, we have $\mathrm{amp}X\geq\mathrm{amp}A$.
\end{prop}
\begin{proof} Let $B=\mathrm{L}\Lambda(A,\bar{\mathfrak{m}})$. Then $X\otimes^\mathrm{L}_AB,Y\otimes^\mathrm{L}_AB\in\mathrm{D}^{\mathrm{b}}_{\mathrm{f}}(B)$ with $\mathrm{projdim}_B(X\otimes^\mathrm{L}_AB)<\infty$ by \cite[Proposition 3.7]{BSSW} and $\mathrm{amp}(X\otimes^\mathrm{L}_AY)=\mathrm{amp}((X\otimes^\mathrm{L}_AB)\otimes^\mathrm{L}_B(Y\otimes^\mathrm{L}_AB))$, $\mathrm{amp}Y=\mathrm{amp}(Y\otimes^\mathrm{L}_AB)$. Thus we may assume that $A$ is derived $\bar{\mathfrak{m}}$-adic complete, and $A$ has a normalized dualizing DG-module $R$ by \cite[Proposition 7.21]{s18}. We have the following isomorphisms
\begin{center}$\begin{aligned}\mathrm{RHom}_A(X\otimes^\mathrm{L}_AY,R)
&\simeq\mathrm{RHom}_A(X,\mathrm{RHom}_A(Y,R))\\
&\simeq
\mathrm{RHom}_A(\mathrm{RHom}_A(\mathrm{RHom}_A(X,A),A),\mathrm{RHom}_A(Y,R))\\
&\simeq\mathrm{RHom}_A(X,A)\otimes^\mathrm{L}_A\mathrm{RHom}_A(Y,R).\end{aligned}$\end{center}
where the second one is by $X\simeq\mathrm{RHom}_A(\mathrm{RHom}_A(X,A),A)$ and the third one is by Lemma \ref{lem0.15}(1) as $\mathrm{projdim}_A\mathrm{RHom}_A(X,A)<\infty$.
So
we have the next (in)equalities
\begin{center}$\begin{aligned}\mathrm{amp}(X\otimes^\mathrm{L}_AY)
&=\mathrm{lc.dim}_A\mathrm{RHom}_A(X\otimes^\mathrm{L}_AY,R)-\mathrm{depth}_A\mathrm{RHom}_A(X\otimes^\mathrm{L}_AY,R)\\
&\geq\mathrm{lc.dim}_A\mathrm{RHom}_A(Y,R)-\mathrm{depth}_A\mathrm{RHom}_A(Y,R)\\
&=\mathrm{amp}Y,\end{aligned}$\end{center}
where the equalities are by Corollary \ref{lem0.8} and the inequality is by the above isomorphism and Corollary \ref{lem1.1}, as desired.
\end{proof}

\begin{rem}\label{lem:2.2}{\rm Let $(A,\bar{\mathfrak{m}},\bar{\kappa})$ be a local DG-ring with $\mathrm{amp}A<\infty$.

(1) If $\mathrm{amp}A>0$, then $\bar{\kappa}$ never has finite projective dimension. Thus there are DG-modules in $\mathrm{D}_{\mathrm{f}}(A)$ which are not compact.

(2) Let $\bar{\textbf{\emph{x}}}$ be a sequence in $\bar{\mathfrak{m}}$ and
$X\in\mathrm{D}^{\mathrm{b}}_{\mathrm{f}}(A)$ with $\mathrm{amp}X=\mathrm{amp}A$. If  $\mathrm{projdim}_AX<\infty$, then $\mathrm{amp}X/\hspace{-0.15cm}/\bar{\emph{\textbf{x}}}=\mathrm{amp}(A/\hspace{-0.15cm}/\bar{\emph{\textbf{x}}}\otimes^\mathrm{L}_AX)
\geq\mathrm{amp}A/\hspace{-0.15cm}/\bar{\emph{\textbf{x}}}\geq\mathrm{amp}A$ by Proposition \ref{lem5.10}, it follows by \cite[Lemma 2.13]{Mi19} that $\mathrm{amp}A/\hspace{-0.15cm}/\bar{\emph{\textbf{x}}}=\mathrm{amp}A$ and $\bar{\emph{\textbf{x}}}$ is $A$-regular. This is the DG-version of Auslander's zerodivisor conjecture.

(3) The example after \cite[Theorem 0.2]{J} shows that the condition $\mathrm{amp}A<\infty$ is crucial.}
\end{rem}

The next result is an extend of the strong amplitude inequality (see \cite[20.8]{F}).

\begin{prop}\label{lem5.11} Let $(A,\bar{\mathfrak{m}},\bar{\kappa})$ be a sequence-regular DG-ring and $X,Y\in\mathrm{D}^{\mathrm{b}}_{\mathrm{f}}(A)$ with $\mathrm{projdim}_AX<\infty$. One has an inequality \begin{center}$\mathrm{amp}(X\otimes^\mathrm{L}_AY)\geq\mathrm{amp}X+\mathrm{amp}Y+\mathrm{inf}A$.\end{center}
\end{prop}
\begin{proof} Let $B=\mathrm{L}\Lambda(A,\bar{\mathfrak{m}})$. It follows by \cite[Corollary 4.5]{sh21} that $A$ is sequence-regular if and only if $B$ is sequence-regular. Also $\mathrm{projdim}_B(X\otimes^\mathrm{L}_AB)<\infty$, we may assume that $A$ is derived $\bar{\mathfrak{m}}$-adic complete and has a dualizing DG-module $R$. Then
$\mathrm{RHom}_A(X\otimes^\mathrm{L}_AY,R)\simeq\mathrm{RHom}_A(\mathrm{RHom}_A(\mathrm{RHom}_A(X,A),A),\mathrm{RHom}_A(Y,R))\simeq \mathrm{RHom}_A(X,A)\otimes^\mathrm{L}_A\mathrm{RHom}_A(Y,R)$ by Lemma \ref{lem0.15}(1) as $\mathrm{projdim}_A\mathrm{RHom}_A(X,A)<\infty$.
Note that $X\simeq\mathrm{RHom}_A(R,R\otimes_A^\mathrm{L}X)$, one has $\mathrm{inf}(R\otimes_A^\mathrm{L}X)\leq\mathrm{inf}X+\mathrm{sup}R$ and $\mathrm{sup}(R\otimes_A^\mathrm{L}X)=\mathrm{sup}X+\mathrm{sup}R$ by Lemma \ref{lem0.2}(3). ALso
$\mathrm{RHom}_A(X,A)\simeq\mathrm{RHom}_A(X\otimes^\mathrm{L}_AR,R)$, we have the next (in)equalities
\begin{center}$\begin{aligned}\mathrm{amp}(X\otimes^\mathrm{L}_AY)
&=\mathrm{lc.dim}_A\mathrm{RHom}_A(X\otimes^\mathrm{L}_AY,R)-\mathrm{depth}_A\mathrm{RHom}_A(X\otimes^\mathrm{L}_AY,R)\\
&=\mathrm{lc.dim}_A\mathrm{RHom}_A(X,A)\otimes^\mathrm{L}_A\mathrm{RHom}_A(Y,R)-
\mathrm{depth}_A\mathrm{RHom}_A(X,A)\otimes^\mathrm{L}_A\mathrm{RHom}_A(Y,R)\\
&\geq\mathrm{amp}(X\otimes^\mathrm{L}_AR)+\mathrm{amp}Y+\mathrm{inf}A\\
&\geq\mathrm{amp}X+\mathrm{amp}Y+\mathrm{inf}A,\end{aligned}$\end{center}where the first inequality is by Proposition \ref{lem5.7}, \cite[Theorem 4.8]{ya20} and Corollary \ref{lem0.8}, the second one is by the preceding proof.
\end{proof}

We denote by $A^\natural$ the graded ring obtained by forgetting the differential of the DG-ring $A$, and
$X^\natural$ denotes the graded $A^\natural$-module obtained by forgetting the differential of $X$ in $\mathrm{D}(A)$.
We now prove the DG-version of NIT and Krull's principle ideal theorem.

\begin{prop}\label{lem1.4} Let $(A,\bar{\mathfrak{m}},\bar{\kappa})$ be a local DG-ring with $\mathrm{amp}A<\infty$.

$(1)$ Let $F$ be a semi-free DG-module with $F^\natural=\coprod_{-n\leq j\leq0}A^\natural[-j]^{(\beta_j)}$, where each $\beta_j$ is finite. If $F\not\simeq0$ and $\mathrm{dim}_{\mathrm{H}^0(A)}\mathrm{H}^i(F)\leq -i$ holds for all $i$, then $\mathrm{lc.dim}A\leq n$.

$(2)$ Let $\bar{x}_1,\cdots,\bar{x}_n$ be a sequence in $\mathrm{H}^{0}(A)$ and $\bar{\mathfrak{p}}$ be minimal among the prime ideals containing $(\bar{x}_1,\cdots,\bar{x}_n)$. Then $\mathrm{lc.dim}A_{\bar{\mathfrak{p}}}\leq n$.
\end{prop}
\begin{proof} (1) By induction on $n$, one can obtain that $\mathrm{projdim}_AF\leq n$. As $\mathrm{dim}_{\mathrm{H}^0(A)}\mathrm{H}^i(F)\leq -i$ for all $i$, one has $\mathrm{lc.dim}_AF\leq0$. Hence Theorem \ref{lem5.4}(2) yields that $n\geq\mathrm{projdim}_AF\geq\mathrm{lc.dim}A-\mathrm{lc.dim}_AF\geq\mathrm{lc.dim}A$.

(2) Let $X=A_{\bar{\mathfrak{p}}}/\hspace{-0.15cm}/(\bar{x}_1/1,\cdots,\bar{x}_n/1)$. Then $\mathrm{projdim}_{A_{\bar{\mathfrak{p}}}}X\leq n$ by \cite[Theorem 2.19]{Mi18}. Also $\mathrm{lc.dim}_{A_{\bar{\mathfrak{p}}}}X=0$, so the asserted inequality follows by Theorem \ref{lem5.4}(2).
\end{proof}

Let $A$ be a local Cohen-Macaulay DG-ring and $\bar{\textbf{\emph{x}}}$ a maximal $A$-regular sequence.
Then the DG-module $X=A/\hspace{-0.15cm}/\bar{\textbf{\emph{x}}}$ has finite projective dimension, $\mathrm{H}(X)$ has finite length and $\mathrm{amp}A=\mathrm{amp}X$.  The next corollary proves that its opposite is also true.

\begin{cor}\label{lem1.3} Let $(A,\bar{\mathfrak{m}},\bar{\kappa})$ be a local DG-ring with $\mathrm{amp}A<\infty$ and $0\not\simeq X\in\mathrm{D}^{\mathrm{b}}(A)$. If $\mathrm{projdim}_AX<\infty$ and $\mathrm{H}(X)$ is finite length, then
$\mathrm{lc.dim}A\leq \mathrm{projdim}_AX+\mathrm{sup}X$.
In additional, if $\mathrm{amp}X=\mathrm{amp}A$ then
$A$ is local Cohen-Macaulay.
\end{cor}
\begin{proof} Note that $\mathrm{projdim}_AX[\mathrm{sup}X]=\mathrm{projdim}_AX+\mathrm{sup}X$ and $\mathrm{sup}X[\mathrm{sup}X]=0$, it follows that $X[\mathrm{sup}X]$ has a minimal semi-free resolution $F$ with $F^\natural=\coprod_{-n\leq j\leq0}A^\natural[-j]^{(\beta_j)}$, where $n=\mathrm{projdim}_AX[\mathrm{sup}X]$ and each $\beta_j$ is finite. As $\mathrm{H}(F)\cong\mathrm{H}(X)$ is finite length, it follows by Proposition \ref{lem1.4}(1) that $\mathrm{lc.dim}A\leq n=\mathrm{projdim}_AX+\mathrm{sup}X$.

If $\mathrm{amp}X=\mathrm{amp}A$, then $\mathrm{lc.dim}A\leq \mathrm{projdim}_AX+\mathrm{sup}X=\mathrm{depth}A-\mathrm{depth}_AX-\mathrm{inf}A+\mathrm{inf}X=\mathrm{depth}A-\mathrm{inf}A$, and so $A$ is local Cohen-Macaulay.
\end{proof}

 The affirmative answer to the DG-version of Bass conjecture is the special case of the following proposition, which solves affirmatively conjecture of Minamoto (see \cite[Conjecture 2.36]{Mi19}) or Shaul (see \cite[Remark 5.25]{Shaul}).

\begin{prop}\label{lem5.9} Let $(A,\bar{\mathfrak{m}},\bar{\kappa})$ be a local DG-ring with $\mathrm{amp}A<\infty$, and let $0\not\simeq X\in\mathrm{D}^{+}_\mathrm{f}(A)$ with $\mathrm{injdim}_AX<\infty$ and $0\not\simeq Y\in\mathrm{D}^{\mathrm{b}}_{\mathrm{f}}(A)$. One has an inequality, \begin{center}$\mathrm{amp}\mathrm{RHom}_A(Y,X)
\geq\mathrm{lc.dim}_AY-\mathrm{depth}_AY$.\end{center}In particular, $A$ is local Cohen-Macaulay if and only if there is such $X$ with $\mathrm{amp}X=\mathrm{amp}A$.
\end{prop}
\begin{proof} Let $B=\mathrm{L}\Lambda(A,\bar{\mathfrak{m}})$. As $\mathrm{H}^0(B)=\mathrm{H}^0(A)\otimes^\mathrm{L}_AB$ by \cite[Lemma 6.3]{S}, we have
\begin{center}$\mathrm{RHom}_B(\mathrm{H}^0(B)/\bar{\mathfrak{m}}\mathrm{H}^0(B),X\otimes^\mathrm{L}_AB)
\simeq\mathrm{RHom}_A(\mathrm{H}^0(A)/\bar{\mathfrak{m}},X)\otimes^\mathrm{L}_AB$,\end{center}
it follows by \cite[Corollary 2.31]{Mi19} that
$\mathrm{injdim}_B(X\otimes^\mathrm{L}_AB)<\infty$.
By \cite[Propositions 2.10 and 3.3 and Theorem 2.15]{Shaul}, one has $\mathrm{depth}_AY=\mathrm{depth}_B(Y\otimes^\mathrm{L}_AB)$ and $\mathrm{lc.dim}_AY=\mathrm{lc.dim}_B(Y\otimes^\mathrm{L}_AB)$. Also $\mathrm{amp}\mathrm{RHom}_A(Y,X)=\mathrm{amp}(\mathrm{RHom}_A(Y,X)\otimes^\mathrm{L}_AB)$ by \cite[Proposition 1.14]{BSSW}.
 Thus
we may assume that $A$ is derived $\bar{\mathfrak{m}}$-adic complete, and $A$ has a normalized dualizing DG-module $R$. Then $\mathrm{RHom}_A(X,R)\in\mathrm{D}^{\mathrm{b}}_{\mathrm{f}}(A)$ by \cite[Proposition 7.2]{ye16} and Lemma \ref{lem0.11}(2) and $\mathrm{projdim}_A\mathrm{RHom}_A(X,R)<\infty$ by Corollary \ref{lem0.8}. Also $\mathrm{RHom}_A(Y,X)\in\mathrm{D}^{\mathrm{b}}_{\mathrm{f}}(A)$ and $\mathrm{RHom}_A(X,R)\otimes_A^\mathrm{L}Y\simeq\mathrm{RHom}_A(\mathrm{RHom}_A(Y,X),R)$, it follows that
\begin{center}$\begin{aligned}\mathrm{lc.dim}_AY-\mathrm{depth}_AY
&\leq\mathrm{lc.dim}_A(\mathrm{RHom}_A(X,R)\otimes_A^\mathrm{L}Y)-
\mathrm{depth}_A(\mathrm{RHom}_A(X,R)\otimes_A^\mathrm{L}Y)\\
&=\mathrm{amp}\mathrm{RHom}_A(Y,X).\end{aligned}$\end{center}
where the inequality is by Corollary \ref{lem1.1} and the equality is by Corollary \ref{lem0.8}.

Let $A$ be local Cohen-Macaulay and $\bar{\textbf{\emph{x}}}$ a maximal $A$-regular sequence.
 Then the DG-module $X=\mathrm{Hom}_A(A/\hspace{-0.15cm}/\bar{\textbf{\emph{x}}},E(A,\bar{\mathfrak{m}}))$ has finite injective dimension and $\mathrm{amp}A=\mathrm{amp}X$. Conversely, one has $\mathrm{amp}A\leq\mathrm{lc.dim}A-\mathrm{depth}A\leq\mathrm{amp}X=\mathrm{amp}A$ by \cite[Theorem 4.1(1)]{Shaul}. Thus $A$ is local Cohen-Macaulay.
\end{proof}

 The next corollary is a DG-setting of Bass characterization of Cohen-Macaulay rings.

\begin{cor}\label{lem5.3} Let $(A,\bar{\mathfrak{m}},\bar{\kappa})$ be a local DG-ring with $\mathrm{amp}A<\infty$. Then the following are equivalent:

 $(1)$ $A$ is local Cohen-Macaulay;

 $(2)$ There exists $0\not\simeq Y\in \mathrm{D}^{\mathrm{b}}_{\mathrm{f}}(A)$ with $\mathrm{amp}Y=\mathrm{amp}A$ so that  $\mathrm{lc.dim}_{A}\mathrm{RHom}_A(X,Y)<\infty$ for all $X\in \mathrm{D}^{\mathrm{b}}_{\mathrm{f}}(A)$;

 $(3)$ There is $0\not\simeq Y\in \mathrm{D}^{\mathrm{b}}_{\mathrm{f}}(A)$ with $\mathrm{amp}Y=\mathrm{amp}A$ so that $\mathrm{lc.dim}_{A}\mathrm{RHom}_A(\bar{\kappa},Y)<\infty$.
\end{cor}
\begin{proof} (1) $\Rightarrow$ (2) By \cite[Theorem 5.22(1)]{Shaul}, there exists $0\not\simeq Y\in \mathrm{D}^{\mathrm{b}}_{\mathrm{f}}(A)$ such that $\mathrm{amp}Y=\mathrm{amp}A$ and $\mathrm{injdim}_AY<\infty$. Hence
 Theorem \ref{lem5.4'}(iii) yields the statement (2).

  (2) $\Rightarrow$ (3) is trivial.

(3) $\Rightarrow$ (1) Since $\mathrm{sup}\mathrm{RHom}_A(\bar{\kappa},Y)=\mathrm{lc.dim}_{A}\mathrm{RHom}_A(\bar{\kappa},Y)<\infty$, $\mathrm{injdim}_AY<\infty$ by \cite[Corollary 2.31]{Mi19}.
Hence Proposition \ref{lem5.9} implies that $A$ is local Cohen-Macaulay.
\end{proof}

  Let $(A,\bar{\mathfrak{m}},\bar{\kappa})$ be a local DG-ring. For $X\in\mathrm{D}(A)$, set
\begin{center} $\mu^n_A(X)=\mathrm{rank}_{\bar{\kappa}}\mathrm{H}^n(\mathrm{RHom}_A(\bar{\kappa},X))$.\end{center}
Vasconcelos \cite{V} conjectured the following result for a local ring.

\begin{prop}\label{lem7.7} Let $(A,\bar{\mathfrak{m}},\bar{\kappa})$ be a local DG-ring with $\mathrm{amp}A<\infty$ and set $d=\mathrm{dim}\mathrm{H}^0(A)$. If $\mu^{d+\mathrm{inf}A}_{A}(A)=1$, then $A$ is local Gorenstein.
\end{prop}
\begin{proof} We use induction on $d$. If $d=0$, then $\mathrm{inf}A=\mathrm{depth}A$ by assumption and \cite[Propositions 3.3]{Shaul}
 and so $A$ is local Cohen-Macaulay. Consequently, $\mathrm{injdim}_AA=\mathrm{supRHom}_A(\bar{\kappa},A)=\mathrm{inf}A$ by \cite[Propositions 4.7]{YL} and $A$ is local Gorenstein. Now assume that $d>0$. By \cite[Theorem 5.18]{Shaul}, there is an $A$-regular element $\bar{x}\in\bar{\mathfrak{m}}$ such that $\mathrm{lc.dim}A/\hspace{-0.15cm}/\bar{x}=\mathrm{dim}\mathrm{H}^0(A)/(\bar{x})=d-1$. Also by \cite[Lemma 2.9]{Mi19}, one has an isomorphism \begin{center}$\mathrm{RHom}_{A/\hspace{-0.1cm}/\bar{x}}(\bar{\kappa},A/\hspace{-0.15cm}/\bar{x})[-1]\simeq\mathrm{RHom}_{A}(\bar{\kappa},A)$\end{center}
 which implies that $\mu^{d-1+\mathrm{inf}A/\hspace{-0.1cm}/\bar{x}}_{A/\hspace{-0.1cm}/\bar{x}}(A/\hspace{-0.15cm}/\bar{x})=1$ as $\mathrm{inf}A/\hspace{-0.15cm}/\bar{x}=\mathrm{inf}A$. By induction, $A/\hspace{-0.15cm}/\bar{x}$ is local Gorenstein, and so $\mathrm{lc.dim}_{A}\mathrm{RHom}_{A/\hspace{-0.1cm}/\bar{x}}(\bar{\kappa},A/\hspace{-0.15cm}/\bar{x})
 =\mathrm{lc.dim}_{A/\hspace{-0.1cm}/\bar{x}}\mathrm{RHom}_{A/\hspace{-0.1cm}/\bar{x}}(\bar{\kappa},A/\hspace{-0.15cm}/\bar{x})<\infty$ by Corollary \ref{lem5.2}. Thus
 $\mathrm{lc.dim}_{A}\mathrm{RHom}_A(\bar{\kappa},A)<\infty$
and $A$ is local Gorenstein.
\end{proof}

\begin{cor}\label{lem8.7} Let $(A,\bar{\mathfrak{m}},\bar{\kappa})$ be a local DG-ring with $\mathrm{amp}A<\infty$ and set $d=\mathrm{dim}\mathrm{H}^0(A)$ and $E=E(A,\bar{\mathfrak{m}})$. Then the following are equivalent:

$(1)$ $A$ is local Gorenstein;

$(2)$ $\mu^{d+\mathrm{inf}A}_{A}(A)=1$;

$(3)$ $\mathrm{RHom}_A(E,\bar{\kappa})\simeq\bar{\kappa}[-d-\mathrm{inf}A]$;

$(4)$ $\bar{\kappa}\otimes_A^\mathrm{L}E\simeq\bar{\kappa}[d+\mathrm{inf}A]$.
\end{cor}
\begin{proof} (1) $\Leftrightarrow$ (2) follows from Proposition \ref{lem7.7} and \cite[Theorem 3.9]{Mi19}.

(1) $\Leftrightarrow$ (4) By \cite[Theorem 3.9]{Mi19} and the isomorphism $\bar{\kappa}\otimes_A^\mathrm{L}E\simeq\mathrm{RHom}_A(\mathrm{RHom}_A(\bar{\kappa},A),E)$.

(3) $\Leftrightarrow$ (4) follows from the isomorphism $\mathrm{RHom}_A(E,\bar{\kappa})\simeq\mathrm{RHom}_A(\bar{\kappa}\otimes_A^\mathrm{L}E,E)$.
\end{proof}

\renewcommand\refname
\end{document}